%% file: OOloid_0.tex
\documentclass[11pt,a4paper]{article}

\usepackage{amsmath}
\usepackage{amssymb,verbatim}
\usepackage{amsfonts}
\usepackage{amsthm}
\usepackage{pifont}
\usepackage{graphicx}
\usepackage[long,hhmmss]{datetime}

\newtheorem{theorem}{Theorem}%[section]
\newtheorem{lemma}{Lemma}

\newtheorem{corollary}{Corollary}
\newtheorem{remark}{Remark}%[section]

\newcommand{\beq}{\begin{equation*}}
\newcommand{\eeq}{\end{equation*}}
\newcommand{\beqn}{\begin{equation}}
\newcommand{\eeqn}{\end{equation}}
\newcommand{\RR}{\mathbb R}
\newcommand{\PP}{\mathbb{P}}
\newcommand{\dd}{\mathrm{d}}
\newcommand{\ii}{\mathrm{i}}

\newcommand{\Q}{\mathcal{Q}_\lambda}
\newcommand{\C}{\mathcal{C}_\lambda}
\newcommand{\Ol}{\mathcal{O}}
\newcommand{\Lt}{\mathcal{L}_t}
\newcommand{\K}{\mathcal{K}}
\newcommand{\R}{\mathcal{R}}
\newcommand{\la}{\lambda}
\newcommand{\T}{\mathcal{T}_\la}
\newcommand{\F}{\mathcal{F}}
\newcommand{\om}{\omega}
\newcommand{\ka}{\kappa}
\newcommand{\ph}{\varphi}
\newcommand{\tta}{\tilde{\tau}}
\newcommand{\A}{\mathcal{A}}
\newcommand{\PT}{\mathcal{P}}
\newcommand{\X}{\mathcal{X}}
\newcommand{\Y}{\mathcal{Y}}
\newcommand{\Z}{\mathcal{Z}}
\newcommand{\g}{\gamma}
\newcommand{\ve}{\varepsilon}
\newcommand{\G}{\mathcal{G}}
\newcommand{\ot}{\widetilde{\om}}

\begin{document}
\input{OOloid_1}
\input{OOloid_2}
\input{OOloid_2a}
\input{OOloid_3}
\input{OOloid_4}
\input{OOloid_5}
\input{OOloid_6}
\input{OOloid_L}
\end{document}

%% file: OOloid_1.tex
% !TeX root = OOloid_0.tex

\title{The extended oloid and its inscribed quadrics}
\author{Uwe B\"asel and Hans Dirnb\"ock}
\date{%\today,\;\currenttime\vspace{-0.5cm}
}
\maketitle

\begin{abstract}
\noindent  The oloid is the convex hull of two circles with equal radius in perpendicular planes so that the center of each circle lies on the other circle. It is part of a developable surface which we call {\em extended oloid}. We determine the tangential system of all inscribed quadrics $\Q$ of the extended oloid $\Ol$ where~$\la$ is the system parameter. From this result we conclude parameter equations of the touching curve $\C$ between $\Ol$ and~$\Q$, the edge of regression $\R$ of $\Ol$, and the asymptotes of $\R$. Properties of the curves $\C$ are investigated, including the case that $\la\rightarrow\pm\infty$. The self-polar tetrahedron of the tangential system $\Q$ is obtained. The common generating lines of $\Ol$ and any ruled surface $\Q$ are determined. Furthermore, we derive the curves which are the images of $\C$ and $\R$ when $\Ol$ is developed onto the plane.\\[0.2cm]
\end{abstract}

\section{Introduction}

The oloid was discovered by Paul Schatz in 1929. It is the convex hull of two circles with equal radius $r$ in perpendicular planes so that the center of each circle lies on the other circle. The oloid has the remarkable properties that it develops its entire surface while rolling, and its surface area is equal to $4\pi r^2$. The surface of the oloid is part of a developable surface. \cite{Dirnboeck_Stachel}, \cite{Wikipedia}

\begin{figure}[h]
\begin{center}
  \includegraphics[scale=0.38]{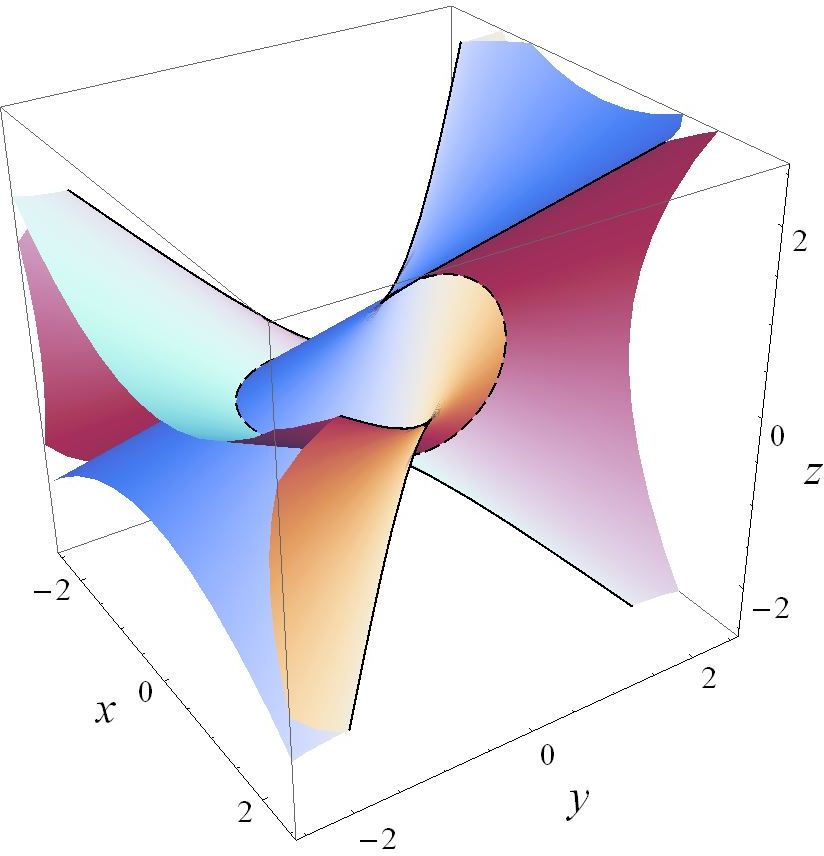}
  \caption{\label{B01} The extended oloid $\Ol$, the circles $k_A$, $k_B$ (dashed lines), and the edge of regression $\R$ (solid lines) in the box $-2.5\leq x,y,z\leq 2.5$}
\end{center}
\vspace{-0.2cm}
\end{figure}

In the following this developable surface is called {\em extended oloid}. According to \cite[pp.\ 105-106]{Dirnboeck_Stachel}, with $r=1$ the circles can be defined by 
\begin{align*}
  k_A := {} & 
  \left\{(x,y,z)\in\RR^3\;\big|\;x^2+\left(y+\tfrac{1}{2}\right)^2 = 1
  \wedge z=0\right\},\\
  k_B := {} & 
  \left\{(x,y,z)\in\RR^3\;\big|\,\left(y-\tfrac{1}{2}\right)^2+z^2 = 1
  \wedge x=0\right\}.  
\end{align*}
In this case we denote the extended oloid by $\Ol$ (see Fig.\ \ref{B01}). 

Now we introduce homogeneous coordinates $x_0,x_1,x_2,x_3$ with
\beq
  x=\frac{x_1}{x_0}\,,\quad y=\frac{x_2}{x_0}\,,\quad z=\frac{x_3}{x_0}\,.
\eeq
Then the real projective space is given by
\beq
  \PP_3(\RR)
  = \left\{[x_0,x_1,x_2,x_3]\mid(x_0,x_1,x_2,x_3)\in\RR^4\setminus\{0\}
    \right\},
\eeq
where $[x_0,x_1,x_2,x_3]=[y_0,y_1,y_2,y_3]$ if there exist a $\mu\in\RR\setminus\{0\}$ such that $x_j=\mu y_j$ for $j=1,2,3$. Should it prove necessary, complex coordinates will be used instead of the real ones. For the description of the corresponding projective circles $\K_A$ and $\K_B$ to $k_A$ and $k_B$, respectively, we write
\begin{align*}
  \K_A := {} & \left\{[x_0,x_1,x_2,x_3]\in\mathbb{P}_3(\RR)\;\big|\;
		 \ph_A(x_0,x_1,x_2,x_3) = 0 \wedge x_3=0\right\},\\[0.05cm]
  \K_B := {} & \left\{[x_0,x_1,x_2,x_3]\in\mathbb{P}_3(\RR)\;\big|\;
		 \ph_B(x_0,x_1,x_2,x_3) = 0 \wedge x_1=0\right\}  
\end{align*}
with
\begin{align*}
  \ph_A(x_0,x_1,x_2,x_3) = {} & 3x_0^2-4x_0x_2-4x_1^2-4x_2^2\,,\\[0.05cm]
  \ph_B(x_0,x_1,x_2,x_3) = {} & 3x_0^2+4x_0x_2-4x_2^2-4x_3^2\,. 
\end{align*}

%% file: OOloid_2.tex
% !TeX root = OOloid_0.tex

\section{Inscribed quadrics}

\begin{lemma}
The dual figures to $\K_A$ and $\K_B$ are
\begin{align*}
  \widehat{\K}_A = {} &  
  \left\{[u_0,u_1,u_2,u_3]\in\mathbb{P}_3(\RR)\;\big|\;
	4u_0^2-4u_0u_2-4u_1^2-3u_2^2 = 0\right\}\quad\mbox{and}\\
  \widehat{\K}_B = {} & 
  \left\{[u_0,u_1,u_2,u_3]\in\mathbb{P}_3(\RR)\;\big|\;
  4u_0^2+4u_0u_2-3u_2^2-4u_3^2 = 0\right\},
\end{align*}
respectively, where $u_0,u_1,u_2,u_3$ are homogeneous plane-coordinates. $\widehat{\K}_A$ and $\widehat{\K}_B$ are elliptic cylinders. Their respective non homogeneous equations are
\beq
  \frac{u^2}{\left(\frac{2}{\sqrt{3}}\right)^2} + 
  \frac{\left(v+\frac{2}{3}\right)^2}{\left(\frac{4}{3}\right)^2} = 1
  \;\wedge\; -\infty<w<\infty\,,
\eeq
and
\beq
  -\infty<u<\infty \;\wedge\;
  \frac{\left(v-\frac{2}{3}\right)^2}{\left(\frac{4}{3}\right)^2} +
  \frac{w^2}{\left(\frac{2}{\sqrt{3}}\right)^2} = 1\,.
\eeq
\end{lemma}

\begin{proof}
Following \cite[pp.\ 41-42]{Fladt_Baur}, \cite[pp.\ 160, 164-165]{Kommerell}, we determine $\widehat{\K}_A$. With 
\beq
  \mu u_i = \frac{\partial\ph_A}{\partial x_i} \,,\quad i\in\{0,1,2,3\}\,,
\eeq
we get
\beq
  \mu u_0 = 6x_0-4x_2 \,,\quad \mu u_1 = -8x_1 \,,\quad 
  \mu u_2 = -4x_0-8x_2 \,,\quad \mu u_3 = 0\,.
\eeq
Solving this system of linear equations for $x_0$, $x_1$, $x_2$ delivers
\beq
  x_0 = \mu\left(\frac{1}{8}\,u_0-\frac{1}{16}\,u_2\right),\quad
  x_1 = \mu\left(-\frac{1}{8}\,u_1\right),\quad
  x_2 = \mu\left(-\frac{1}{16}\,u_0-\frac{3}{32}\,u_2\right).
\eeq
So we get
\beq
  0 = u_0x_0+u_1x_1+u_2x_2+u_3x_3 =
  \frac{1}{32}\,\mu\left(4u_0^2-4u_0u_2-4u_1^2-3u_2^2\right);
\eeq
hence
\beq
  \widehat{\K}_A :=  
  \left\{[u_0,u_1,u_2,u_3]\in\mathbb{P}_3(\RR)\;\big|\;
  4u_0^2-4u_0u_2-4u_1^2-3u_2^2 = 0\right\}.
\eeq 
Analogously, one finds $\widehat{\K}_B$ from $\K_B$. With $u_0=1$, $u_1=u$, $u_2=v$, $u_3=w$, the non homogeneous polynomials for $\widehat{\K}_A$ and $\widehat{\K}_B$ follow immediately. 
\end{proof}

\begin{remark}
{\em The equations $4u_0^2-4u_0u_2-4u_1^2-3u_2^2=0$ and $4u_0^2+4u_0u_2-3u_2^2-4u_3^2=0$ of $\widehat{\K}_A$ and $\widehat{\K}_B$, respectively, were already given in \cite[p.\ 115]{Dirnboeck_Stachel}.}
\end{remark}

\begin{theorem}
The inscribed quadrics of the extended oloid $\Ol$ are given by
\beq
 \Q = \left\{(x,y,z)\in\RR^3\;\big|\;f_\lambda(x,y,z)=0\right\}   
\eeq
with
\beq
  f_\lambda(x,y,z)
  = \frac{x^2}{1-\lambda}+\frac{\left(y-\lambda+\frac{1}{2}\right)^2}
	{1-\lambda+\lambda^2}+\frac{z^2}{\lambda}-1\,.
\eeq
\end{theorem}

\begin{proof}
For abbreviation we put
\begin{align*}
  F_0(\bar{u}) := {} & 4u_0^2-4u_0u_2-4u_1^2-3u_2^2\,,\quad
  F_1(\bar{u}) := 4u_0^2+4u_0u_2-3u_2^2-4u_3^2 
\end{align*}
with $\bar{u}:=[u_0,u_1,u_2,u_3]$. Then
\beq
  \widehat{\F}_\la := \left\{\bar{u}\in\mathbb{P}_3(\RR)\;|\;
  F_\la(\bar{u})=0\right\}
\eeq
with
\beq
  F_\la(\bar{u}):=(1-\lambda)\,F_0(\bar{u})+\lambda\,F_1(\bar{u})
\eeq
defines a tangential system of quadrics in plane-coordinates (cp.\ \cite[p.\ 253]{Sommerville}). Now we shall determine the point-coordinate representation $\F_\la$ of $\widehat{\F}_\la$. Due to duality (see \cite[p.\ 163]{Kommerell}) we have
\beq
  \sigma x_j
  = \frac{\partial F_\lambda(\bar{u})}{\partial u_j}
  = (1-\lambda)\,\frac{\partial F_0(\bar{u})}{\partial u_j}
  + \lambda\,\frac{\partial F_1(\bar{u})}{\partial u_j}\,,\quad
  j=0,1,2,3.
\eeq
The calculation of the partial derivatives yields
\begin{align*}
  \sigma x_0
  = {} & (1-\lambda)(8u_0-4u_2)+\lambda(8u_0+4u_2)
  = 8u_0+4(2\lambda-1)u_2\,,\\
  \sigma x_1 
  = {} & \!-\!8(1-\lambda)u_1\,,\\
  \sigma x_2
  = {} & (1-\lambda)(-4u_0-6u_2)+\lambda(4u_0-6u_2)
  = 4(2\lambda-1)u_0-6u_2\,,\\
  \sigma x_3
  = {} & -8\lambda u_3\,.
\end{align*}
Solving this system of linear equations for $u_0,\ldots,u_3$ delivers
\begin{align*}
  u_0
  = {} & \sigma\,\frac{3x_0-2(1-2\la)x_2}{32(1-\la+\la^2)}
		\,,\quad\:\, u_1 = -\sigma\,\frac{x_1}{8(1-\la)}\,,\\
  u_2
  = {} & \!-\!\sigma\,\frac{(1-2\la)x_0+2x_2}{16(1-\la+\la^2)}
		\,,\quad u_3 = -\sigma\,\frac{x_3}{8\la}\,.
\end{align*}
So we find
\begin{align*}
0 = \sum_{i=0}^3 x_iu_i
  = {} & \!-\!\frac{\sigma}{32}\left(\frac{4x_1^2}{1-\la}
		+ \frac{4\left(x_2^2+(1-2\la)x_0x_2\right)-3x_0^2}{1-\la+\la^2}
		+ \frac{4x_3^2}{\la}\right)\\
  = {} & \!-\!\frac{\sigma}{8}\left(\frac{x_1^2}{1-\la}
		+ \frac{\left(x_2+\left(\frac{1}{2}-\la\right)x_0\right)^2}
		{1-\la+\la^2} + \frac{x_3^2}{\la} - x_0^2\right),
\end{align*}
and therefore
\beq
  \F_\la = \big\{[x_0,x_1,x_2,x_3]\in\mathbb{P}_3(\RR)\;\big|\;
  \tilde{f}_\la(x_0,x_1,x_2,x_3)=0\big\}
\eeq
with
\beqn \label{tilde(f)}
  \tilde{f}_\la(x_0,x_1,x_2,x_3) = \frac{x_1^2}{1-\la} 
	+ \frac{\left(x_2+\left(\frac{1}{2}-\la\right)x_0\right)^2}
	{1-\la+\la^2} + \frac{x_3^2}{\la} - x_0^2\,.
\eeqn
Finally, we write the representation of $\F_\la$ in non homogeneous coordinates $x,y,z$ with $x_0=1$, $x_1=x$, $x_2=y$, $x_3=z$ as
\beq
 \Q = \left\{(x,y,z)\in\RR^3\;\big|\;f_\la(x,y,z)=0\right\},   
\eeq
where
\beq
  f_\la(x,y,z)
  = \frac{x^2}{1-\la}+\frac{\left(y-\la+\frac{1}{2}\right)^2}
	{1-\la+\la^2}+\frac{z^2}{\la}-1\,.\qedhere
\eeq
\end{proof}
\noindent Now we classify the quadrics $\Q$ with real parameter $\la$ in Euclidean space. We start with the case that $\la$ tends to $\pm\infty$. One finds
\beq
  \lim_{\la\rightarrow\pm\infty}f_\la(x,y,z) = 0\,.
\eeq
So we consider $\la\cdot f_\la$ instead of $f_\la$ and find
\begin{align*}
 & \lim_{\la\rightarrow\pm\infty}\la\cdot\frac{x^2}{1-\la}
	= \lim_{\la\rightarrow\pm\infty}\frac{x^2}{\frac{1}{\la}-1}
	= -x^2\,,\quad
 \lim_{\la\rightarrow\pm\infty}\la\cdot\frac{z^2}{\la} = z^2\,,\\[0.3cm]
 & \lim_{\la\rightarrow\pm\infty}\la\cdot
	\left(\frac{\left(y-\la+\frac{1}{2}\right)^2}{1-\la+\la^2}-1\right)
	= \lim_{\la\rightarrow\pm\infty}
	  \frac{\la\left(y^2+y-\frac{3}{4}\right)-2\la^2y}{1-\la+\la^2}\\
 &	\hspace{5.1cm} = \lim_{\la\rightarrow\pm\infty}
	  \frac{\frac{1}{\la}\left(y^2+y-\frac{3}{4}\right)-2y}
	  {\frac{1}{\la^2}-\frac{1}{\la}+1}
	= -2y\,; 
\end{align*}
hence
\beq
  \lim_{\la\rightarrow\pm\infty}\la\,f_\la(x,y,z) = -x^2+z^2-2y\,.
\eeq
In order to abbreviate notation we put $a^2:=|1-\la|$, $b^2:=1-\la+\la^2>0$ for every $\lambda\in\RR$, $c^2:=|\la|$. So we have the quadrics $\Q$ in the following table:
\begin{center}
\begin{tabular}{|@{\,}c@{\;\,}|c|p{4.1cm}|}\hline
\rule{0pt}{14pt}
$\la=-\infty$ & $x^2-z^2+2y=0$ & Hyperbolic paraboloid\\[3pt] \hline
\rule{0pt}{22.5pt}
$\RR\ni\la<0$ & $\dfrac{x^2}{a^2}+\dfrac{\left(y-\la+\frac{1}{2}\right)^2}{b^2}-\dfrac{z^2}{c^2}=1$ &
Hyperboloid of one sheet\\[10pt] \hline
\rule{0pt}{15pt}
$\la=0$ & $x^2+\left(y+\tfrac{1}{2}\right)^2 = 1 \,\wedge\, z=0$ & 
Circle $k_A$\\[4pt] \hline
\rule{0pt}{22.5pt}
$0<\la<1$ & $\dfrac{x^2}{a^2}+\dfrac{\left(y-\la+\frac{1}{2}\right)^2}{b^2}+\dfrac{z^2}{c^2}=1$ &
Ellipsoid\\[10pt] \hline
\rule{0pt}{15pt}
$\la=1$ & $\left(y-\tfrac{1}{2}\right)^2+z^2 = 1 \,\wedge\, x=0$ &
Circle $k_B$\\[4pt] \hline
\rule{0pt}{22.5pt}
$\RR\ni\la>1$ & $\dfrac{\left(y-\la+\frac{1}{2}\right)^2}{b^2}+\dfrac{z^2}{c^2}-\dfrac{x^2}{a^2}=1$ &
Hyperboloid of one sheet\\[10pt] \hline
\rule{0pt}{14pt}
$\la=\infty$ & $x^2-z^2+2y=0$ & Hyperbolic paraboloid\\[3pt] \hline
\end{tabular}
\end{center}

\noindent
A point of the circle $k_A$ is given by
\beq
  A = \big(\alpha_1(t),\alpha_2(t),\alpha_3(t)\big) 
\eeq
with
\beq
  \alpha_1(t)=\sin t\,,\quad \alpha_2(t)=-\frac{1}{2}-\cos t\,,\quad
  \alpha_3(t)=0\,.
\eeq 
There are two points $B_1$, $B_2$ of the circle $k_B$ which have common generating lines $AB_1$ and $AB_2$, respectively, with $A$:
\beq
  B_1 = \big(\beta_1(t),\beta_2(t),\beta_3(t)\big)\,,\quad
  B_2 = \big(\beta_1(t),\beta_2(t),-\beta_3(t)\big)\,, 
\eeq
where
\beq
  \beta_1(t) = 0\,,\quad 
  \beta_2(t) = \frac{1}{2}-\frac{\cos t}{1+\cos t}\,,\quad
  \beta_3(t) = \frac{\sqrt{1+2\cos t}}{1+\cos t}
\eeq 
(see \cite[pp.\ 106-107]{Dirnboeck_Stachel}). Hence, for fixed $t\in [-2\pi/3,2\pi/3]$, parametric functions of a line $AB_1$ are
\begin{align*}
  \om_i(m,t) := {} & (1-m)\alpha_i(t)+m\beta_i(t)\,,\quad m\in\RR\,,\quad
  i = 1,2,3\,.
\end{align*}
One finds
\beqn \label{omega_i}  
 \left.\begin{aligned}
  \om_1(m,t) = {} & (1-m)\sin t,\\[0.1cm]
  \om_2(m,t) = {} & \frac{2(m-1)\cos^2 t+(2m-3)\cos t+2m-1}
				{2(1+\cos t)}\,,\\
  \om_3(m,t) = {} &\frac{m\,\sqrt{1+2\cos t}}{1+\cos t}\,.
 \end{aligned}\quad\right\}
\eeqn
It follows that
\beq \label{PG} \hspace{-0.3cm}\left.
\begin{array}{l@{\;\,}*2{l@{\;=\;}r@{\,,\;}}l@{\;=\;}r}
1) & x &  \om_1(m,t) & y &  \om_2(m,t) & z &  \om_3(m,t)\,,\\[0.12cm]
2) & x &  \om_1(m,t) & y &  \om_2(m,t) & z & -\om_3(m,t)\,,\\[0.12cm]
\end{array}
\right\}\;
t\in\left[-\dfrac{2\pi}{3},\dfrac{2\pi}{3}\right],\; m\in\RR\,,
\eeq
are the parametric equations of all generating lines of $\Ol$, hence a parametrisation of $\Ol$. The restriction of the parameter $m$ to the interval $[0,1]$ yields the oloid in the narrow sense as the convex hull of $k_A$ and $k_B$.

In the following, we need the intervals
\beqn \label{I_i}
  I_1 := \left[-\frac{2\pi}{3},\,\frac{2\pi}{3}\right],\quad
  I_2 := \left(\frac{2\pi}{3},\,2\pi\right],
\eeqn
and the planes
\begin{align*}
  \X := {} & \{(x,y,z)\in\RR^3\,|\,x=0\}\,,\quad
  \Y := \{(x,y,z)\in\RR^3\,|\,y=0\}\,,\\
  \Z := {} & \{(x,y,z)\in\RR^3\,|\,z=0\}\,.
\end{align*}

\begin{corollary} \label{touching_curve}
For fixed value of $\la\in\RR$, a parametrization of the touching curve $\C$ between $\Ol$ and $\Q$ is given by 
\begin{align*}
\;\, & \gamma(\la,\cdot) : \;
  I_1\cup I_2 \rightarrow \RR^3\,,\quad
  t\; \mapsto \gamma(\la,t)
  =\left\{
  \begin{array}{l@{\quad\mbox{if}\quad}l}
	\gamma_1(\la,t) & t\in I_1\,,\\[0.2cm]
	\gamma_2(\la,t) & t\in I_2\,, 
  \end{array}\right.  
\end{align*}
with
\begin{align*}
  \gamma_1(\la,t) = {} & 
	\big(\ka_1(\la,t),\,\ka_2(\la,t),\,\ka_3(\la,t)\big),\\[0.1cm]
  \gamma_2(\la,t) = {} & \!
	\left(\ka_1\!\left(\la,\,\frac{4\pi}{3}-t\right)\!,\;
	\ka_2\!\left(\la,\,\frac{4\pi}{3}-t\right)\!,\;
	-\ka_3\!\left(\la,\,\frac{4\pi}{3}-t\right)\!\right),
\end{align*}
where
\begin{align*}
  \ka_1(\la,t)
  = {} & \frac{(1-\la)\sin t}{1+\la\cos t}\,,\quad
  \ka_2(\la,t)
  = \frac{2\la-1+(\la-2)\cos t}{2(1+\la\cos t)}\,,\\[0.1cm]
  \ka_3(\la,t)
  = {} & \frac{\la\,\sqrt{1+2\cos t}}{1+\la\cos t}\,.
\end{align*}
\end{corollary}
\noindent
\begin{proof}
For fixed value of $t$ the generating line
\beq
  \Lt := \left\{(x,y,z)\in\RR^3\;\big|\;
  x=\om_1(m,t),\, y=\om_2(m,t),\, z=\om_3(m,t);\, m\in\RR\right\}
\eeq
of $\Ol$ is tangent to $\Q$ for one value $\widetilde{m}$ of $m$. As double solution of the equation
\beq
  f_\la(\om_1(m,t),\om_2(m,t),\om_3(m,t)) = 0
\eeq
one finds
\beq
  \widetilde{m} = \psi(\la,t) := \frac{\la(1+\cos t)}{1+\la\cos t}\,.
\eeq
It follows that
\begin{align*}
  \om_1(\psi(\la,t),t) = {} & \frac{(1-\la)\sin t}{1+\la\cos t}\,,\quad
  \om_2(\psi(\la,t),t) = \frac{2\la-1+(\la-2)\cos t}{2(1+\la\cos t)}\,,
  \displaybreak[0]\\[0.1cm]
  \om_3(\psi(\la,t),t) = {} & \frac{\la\,\sqrt{1+2\cos t}}{1+\la\cos t}\,.
\end{align*}
We put $\ka_j(\la,t):=\om_j(\psi(\la,t),t)$, $j=1,2,3$. This yields
\beq
  \gamma_1(\la,t):=\big(\ka_1(\la,t),\,\ka_2(\la,t),\,\ka_3(\la,t)\big)
\eeq
as contact point of $\Lt$ and $\Q$ for all lines $\Lt$ with $t\in I_1$. Due to the symmetry of $\Ol$ with respect to the plane $\Z$, we have
\beq
  \gamma_2(\la,t) :=
  \left(\ka_1\!\left(\la,\,\tfrac{4\pi}{3}-t\right)\!,\,
        \ka_2\!\left(\la,\,\tfrac{4\pi}{3}-t\right)\!,\,
      - \ka_3\!\left(\la,\,\tfrac{4\pi}{3}-t\right)\right)
\eeq
if $t\in I_2$. Obviously,
\beq
  \gamma_2(\la,2\pi/3)=\gamma_1(\la,2\pi/3)\quad\mbox{and}\quad 
  \gamma_2(\la,2\pi)=\gamma_1(\la,-2\pi/3)
\eeq 
for every $\la\in\RR$.
\end{proof}
\noindent
Examples with inscribed quadric and touching curves are shown in Fig.\ \ref{B02} and Fig.\ \ref{B03}.
\begin{remark}
{\em In the special case $\la=1/2$ one gets the equations
\beq
  x = \frac{\sin t}{2+\cos t}\,,\quad
  y = -\frac{3\cos t}{2(2+\cos t)}\,,\quad
  z = \pm\frac{\sqrt{1+2\cos t}}{2+\cos t}
\eeq
of the inscribed ellipsoid in \cite[p.\ 115, Eq.\ (27)]{Dirnboeck_Stachel}.}
\end{remark}

%% file: OOloid_2a.tex
% !TeX root = OOloid_0.tex

\section{Properties of the touching curves $\C$}

Since $f_\la(-x,y,z)$ $=f_\la(x,y,z)$ and $f_\la(x,y,-z)=f_\la(x,y,z)$, every quadric $\Q$ is symmetric with respect to $\X$ and $\Z$. The extended oloid $\Ol$ is symmetric with respect to these planes, too. It follows that all touching curves $\C$ are symmetric with respect to $\X$ and $\Z$. We denote by $X_1$, $X_2$ the intersection points of $\C$ and $\X$, and by $Z_1$, $Z_2$ those of $\C$ and $\Z$, and find
\beqn \label{X_iZ_i}
\left.
\begin{aligned}
 X_1 = X_1(\la)
  = {} & \g(\la,0)
  = \left(0,\,-\frac{3(1-\la)}{2(1+\la)},\,
		\frac{\sqrt{3}\,\la}{1+\la}\right),\\[0.1cm]
 Z_1 = Z_1(\la)
  = {} & \g\left(\la,\,\frac{2\pi}{3}\right)
  = \left(\frac{\sqrt{3}\,(1-\la)}{2-\la},\,
		\frac{3\la}{2(2-\la)},\,0\right),\\[0.1cm]
 X_2 = X_2(\la) 
  = {} & \g\left(\la,\,\frac{4\pi}{3}\right)
  = \left(0,\,-\frac{3(1-\la)}{2(1+\la)},\,-
		\frac{\sqrt{3}\,\la}{1+\la}\right),\\[0.1cm]
 Z_2 = Z_2(\la)
  = {} & \g(\la,2\pi)
  = \left(-\frac{\sqrt{3}\,(1-\la)}{2-\la},\,
		\frac{3\la}{2(2-\la)},\,0\right).
\end{aligned}
\;\;\right\}
\eeqn
One easily finds the parametrization $\T$ for the tangent of $\C$ in the point $\g(\la,t)$:
\beqn \label{T_la}
  \T(t) = \left\{\!\!
  \begin{array}{ll}
	\big\{(x,y,z)\in\RR^3\:\big|\:x=\tau_1(\la,t,\mu),\,y=\tau_2(\la,t,\mu),\\[0.12cm]
	\hspace{0.3cm}z=\tau_3(\la,t,\mu)\,;\,\mu\in\RR\big\}\hspace{1.67cm}\mbox{if}\quad t\in I_1\,,\\[0.35cm]
	\left\{(x,y,z)\in\RR^3\:\big|\:
	 x = \tau_1\!\left(\la,\,\tfrac{4\pi}{3}\!-\!t,\,\mu\right)\!,\,
	 y = \tau_2\!\left(\la,\,\tfrac{4\pi}{3}\!-\!t,\,\mu\right)\!,\right.\\[0.12cm]
	 \hspace{0.3cm}
	 z = \left.-\tau_3\!\left(\la,\,\tfrac{4\pi}{3}\!-\!t,\,\mu\right)\!;\,\mu\in\RR\right\}
	 \quad\mbox{if}\quad t\in I_2\,,	
  \end{array}\hspace{-0.15cm}\right\}
\eeqn
with
\beq
  \tau_j(\la,t,\mu) = \ka_j(\la,t)+\mu\,\dot{\ka}_j(\la,t)\,,\quad
  \dot{\ka}_j = \frac{\dd\ka_j(\la,t)}{\dd t}\,,\quad j=1,2,3\,,
\eeq
where
\begin{align*}
  \dot{\ka}_1(\la,t)
  = {} & \frac{(1-\la)(\la+\cos t)}{(1+\la\cos t)^2}\,,\quad
  \dot{\ka}_2(\la,t)
  = \frac{(1-\la+\la^2)\sin t}{(1+\la\cos t)^2}\,,\\[0.1cm]
  \dot{\ka}_3(\la,t)
  = {} & \frac{\la[\la(1+\cos t)-1]\sin t}{(1+\la\cos t)^2\,\sqrt{1+2\cos t}}\,.  
\end{align*}

\begin{figure}[H]
\begin{center}
  \includegraphics[scale=0.37]{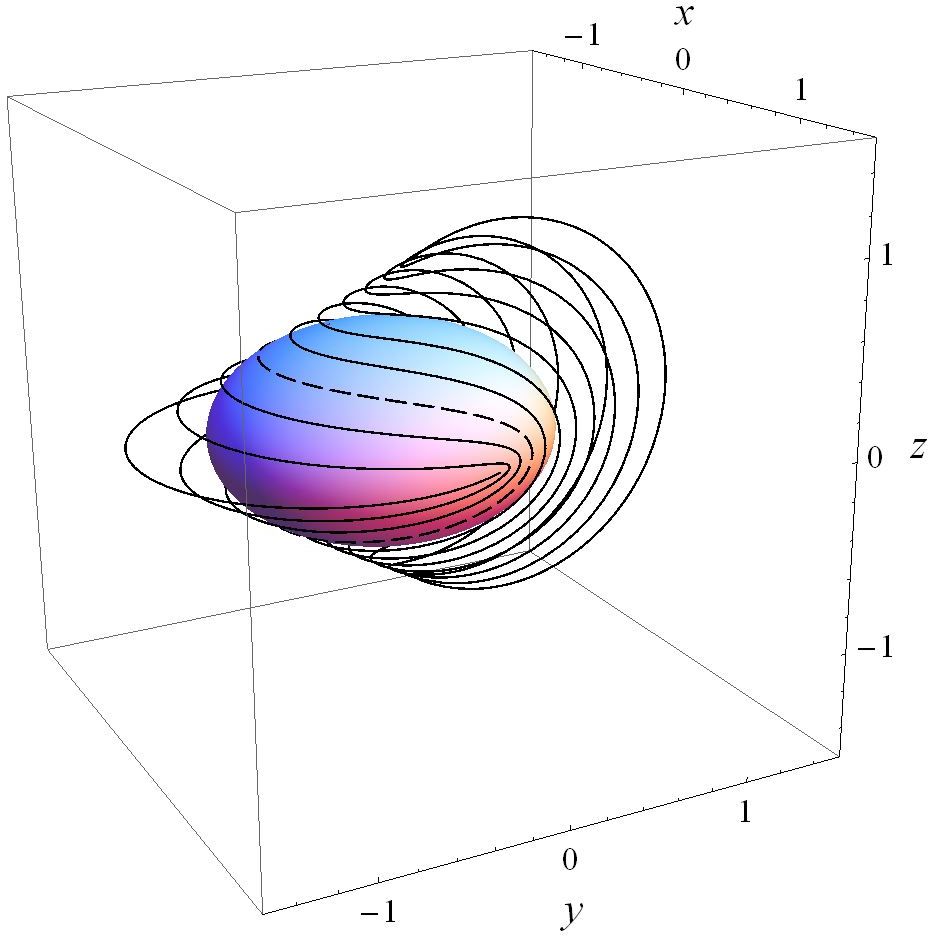}
\end{center}
\vspace{-0.5cm}
\caption{\label{B02} Touching curves $\C$, $\la=0,\,0.1,\,0.2,\ldots,\,0.9,\,1$, and ellipsoid $\mathcal{Q}_{0.3}$ in the box $-1.5\leq x,y,z\leq 1.5$; $\mathcal{C}_{0.3}$ with dashed line}
\end{figure}
%\vspace{1.3cm}
\begin{figure}[h]
\begin{minipage}[c]{0.5\textwidth}
\includegraphics[width=\textwidth]{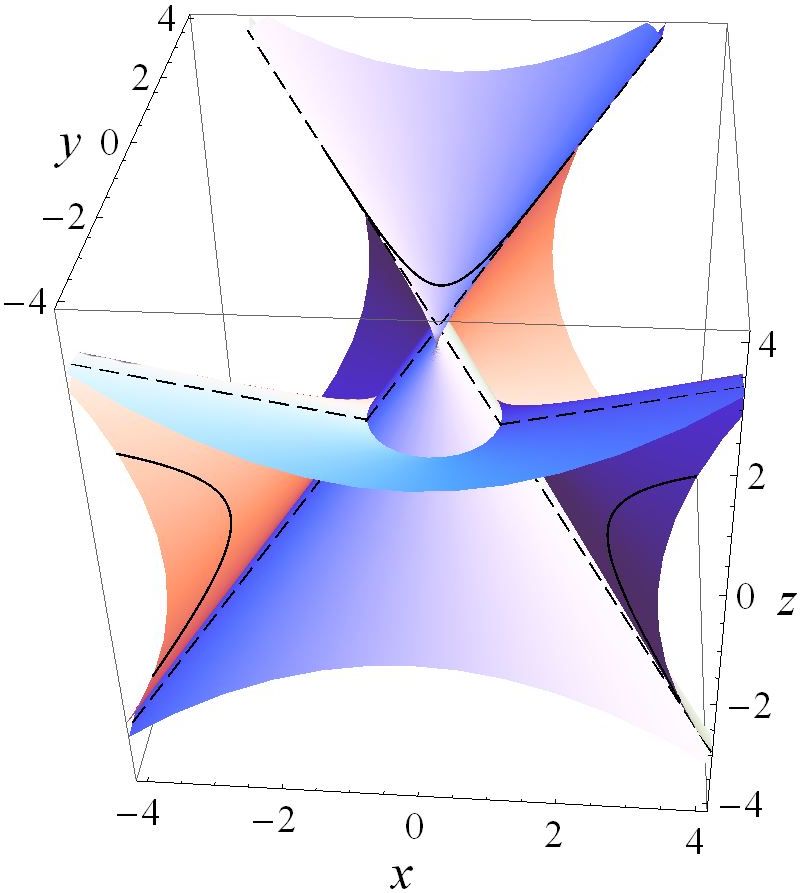}
\end{minipage}
\begin{minipage}[c]{0.5\textwidth}
\includegraphics[width=\textwidth]{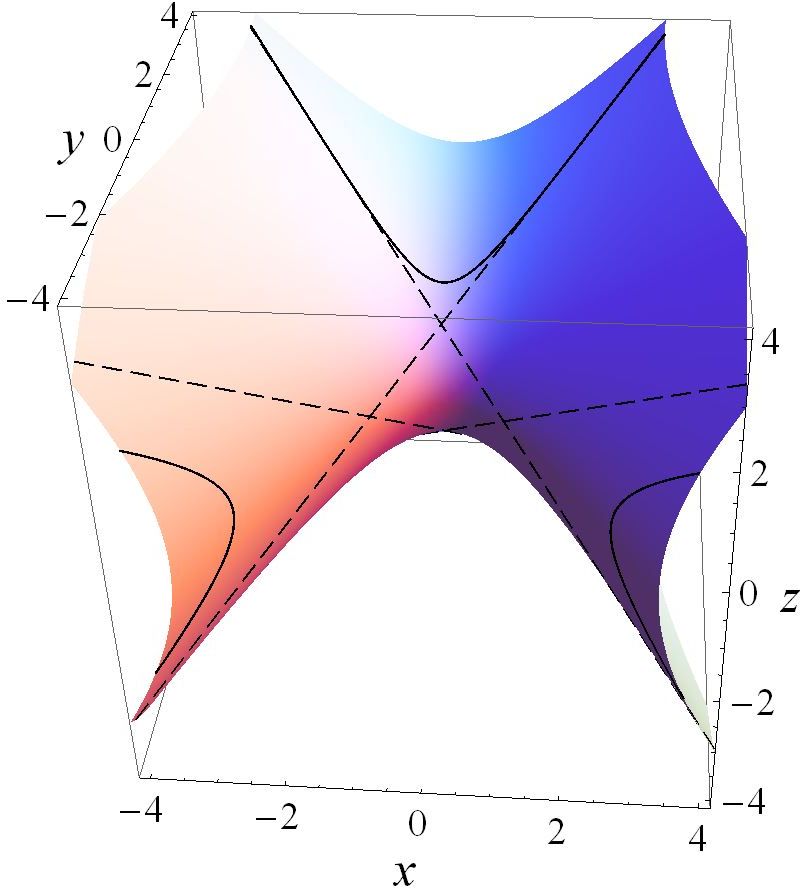}
\end{minipage}
\caption{\label{B03} Extended oloid $\Ol$ (left) and hyperboloid $\mathcal{Q}_4$ (right) with touching curve~ $\mathcal{C}_4$ (solid lines) and common generating lines $\G_j(4)$, $j=1,2,3,4$, (dashed) in the box $-4\leq x,y,z\leq 4$}
\end{figure}

\clearpage

We consider the function $1+\la\cos t$ in the denominator of $\ka_1,\ka_2,\ka_3$. It vanishes in the interval $I_1$ for $t=\pm\arccos(-1/\la)$ if $\la\in\RR\setminus(-1,2)$. (For $\la=-1$, we have $t=0$; and for $\la=2$, $t=\pm 2\pi/3$.) $1+\la\cos t$ has no zeros in $I_1$ if $\la\in(-1,2)$. Therefore, $\ka_1,\ka_2,\ka_3$ are continuous functions if $\la\in(-1,2)$; they are not continous if $\la\in\RR\setminus(-1,2)$. So we have to distinguish the following cases:\\[0.2cm]
\underline{Case 1: $-1<\la<2$}
Since $\ka_1,\ka_2,\ka_3$ are continuous functions, and
\beq
  \g_2(\la,2\pi/3) = Z_1(\la) = \g_1(\la,2\pi/3),\quad
  \g_2(\la,2\pi)   = Z_2(\la) = \g_1(\la,-2\pi/3),
\eeq
the curve $\C$ is closed (see Fig.\ \ref{B02}). As an example with $\la\notin[0,1]$, Fig.\ \ref{C-0.87xyz} shows two projections of $\mathcal{C}_{-0.87}$. The projection of $\mathcal{C}_{-0.87}$ onto the plane $\X$ (thick line in the diagram on the left of Fig.\ \ref{C-0.87xyz}) is part of the left branch of the hyperbola with center point $y=29831/49938\approx 0.597$, $z=0$.\\[0.2cm]
\underline{Case 2: $\la\in\RR\setminus[-1,2]$} The parametrization $\g(\la,t)$ of $\C$ has poles for
\begin{align*}
  t_1 = {} & \!-\!\arccos\left(\!-\frac{1}{\la}\right),\hspace{0.85cm} 
  t_2 = \arccos\left(\!-\frac{1}{\la}\right),\\
  t_3 = {} & \frac{4\pi}{3}-\arccos\left(\!-\frac{1}{\la}\right),\quad
  t_4 = \frac{4\pi}{3}+\arccos\left(\!-\frac{1}{\la}\right).
\end{align*}
Therefore, it consists of four branches. We calculate the asymptotes of $\C$ in the poles. For $t\in I_1$ the tangent $\T(t)$ intersects $\X$ in the point with the coordinates
\begin{align*}
  y = {} & \ka_{2}(\la,t)-\frac{\dot{\ka}_2(\la,t)}{\dot{\ka}_1(\la,t)}\,\ka_1(\la,t)
    = \frac{\la-2+(2\la-1)\cos t}{2(\la+\cos t)}\,,\displaybreak[0]\\[0.1cm]
  z = {} & \ka_{3}(\la,t)-\frac{\dot{\ka}_3(\la,t)}{\dot{\ka}_1(\la,t)}\,\ka_1(\la,t)
    = \frac{\la(1+\cos t+\cos^2 t)}{(\la+\cos t)\,\sqrt{1+2\cos t}}\,,		
\end{align*}
and $\Z$ in
\begin{align*}
  x = {} & \ka_{1}(\la,t)-\frac{\dot{\ka}_1(\la,t)}{\dot{\ka}_3(\la,t)}\,\ka_3(\la,t)
    = \frac{(\la-1)(1+\cos t+\cos^2 t)}{[\la(1+\cos t)-1]\sin t}\,,\\[0.1cm]
  y = {} & \ka_{2}(\la,t)-\frac{\dot{\ka}_2(\la,t)}{\dot{\ka}_3(\la,t)}\,\ka_3(\la,t)
    = -\frac{1+\la+(2-\la)\cos t}{2[\la(1+\cos t)-1]}\,.		
\end{align*}
For $t=t_1$ one finds that the asymptote $\T(t_1)$ intersects $\Z$ in the point
\beq
  (x_1,y_1,z_1) =
  \left(
	\frac{1-\la+\la^2}{2+\la-\la^2}\,\sqrt{1-\frac{1}{\la^2}}\,,\;
	\frac{2-2\la-\la^2}{2\la(\la-2)}\,,\;
	0
  \right),
\eeq
and $\X$ in
\beq
  (\tilde{x}_1,\tilde{y}_1,\tilde{z}_1) =
  \left(
	0\,,\;
	\frac{1-4\la+\la^2}{2(\la^2-1)}\,,\;
	\frac{1-\la+\la^2}{\la^2-1}\,\sqrt{\frac{\la}{\la-2}}\,\right). 
\eeq
Hence, a parametrization $\A_1(\la)$ of the asymptote $\T(t_1)$ is
\beq
  \A_1(\la) = \{(x,y,z)\in\RR^3\:\big|\:x=\tta_1(\la,\nu),\,
	y=\tta_2(\la,\nu),\,z=\tta_3(\la,\nu);\,\nu\in\RR\}
\eeq
with
\begin{align*}
  \tta_1(\la,\nu) := x_1+\nu\,(\tilde{x}_1-x_1)
  = {} & \frac{1-\la+\la^2}{2+\la-\la^2}\,\sqrt{1-\frac{1}{\la^2}}\;\,(1-\nu)\,,\\[0.1cm]
  \tta_2(\la,\nu) := \,y_1+\nu\,(\tilde{y}_1-y_1)
  = {} & \frac{2-2\la-\la^2}{2\la(\la-2)}
		+\nu\,\frac{(1-\la+\la^2)^2}{\la(\la-2)(\la^2-1)}\,,\\
  \tta_3(\la,\nu) := \,z_1+\nu\,(\tilde{z}_1-z_1)
  = {} & \nu\,\frac{1-\la+\la^2}{\la^2-1}\,\sqrt{\frac{\la}{\la-2}}\,.
\end{align*}
In order to abbreviate notation we write
\beq
  \A_1(\la) = \{(\tta_1(\la,\nu),\,\tta_2(\la,\nu),\,
	\tta_3(\la,\nu))\:|\:\nu\in\RR\}\,.
\eeq
The remaining asymptotes are
\begin{align*}
  t=t_2\!:\quad \A_2(\la)
  = {} & \{(-\tta_1(\la,\nu),\,\tta_2(\la,\nu),\,\tta_3(\la,\nu))\:|\:\nu\in\RR\}\,,\\[0.1cm]
  t=t_3\!:\quad \A_3(\la)
  = {} & \{(-\tta_1(\la,\nu),\,\tta_2(\la,\nu),\,-\tta_3(\la,\nu))\:|\:\nu\in\RR\}\,,\\[0.1cm]
  t=t_4\!:\quad \A_4(\la)
  = {} & \{(\tta_1(\la,\nu),\,\tta_2(\la,\nu),\,-\tta_3(\la,\nu))\:|\:\nu\in\RR\}\,.  
\end{align*}
As an example, Fig.\ \ref{C-1.4xyz} shows two projections of the touching curve $\mathcal{C}_{-1.4}$. The projection onto the plane $\X$ (thick line) is part of a hyperbola with center point $y=71/238\approx 0.298$, $z=0$.\\[0.3cm]
\underline{Case 3: $\la=\pm\infty$}
For $\g^*(t):=\lim_{\la\rightarrow\pm\infty}\g(\la,t)$ one easily finds
\begin{align*}
\;\, & \g^* : \;
  I_1\cup I_2 \rightarrow \RR^3\,,\quad
  t\; \mapsto \g^*(t)
  =\left\{
  \begin{array}{l@{\quad\mbox{if}\quad}l}
	\g_1^*(t) & t\in I_1\,,\\[0.2cm]
	\g_2^*(t) & t\in I_2\,, 
  \end{array}\right.  
\end{align*}
with
\begin{align*}
  \g_1^*(t) = {} & 
	\big(\ka_1^*(t),\,\ka_2^*(t),\,\ka_3^*(t)\big),\\[0.1cm]
  \g_2^*(t) = {} & \!
	\left(\ka_1^*\!\left(\frac{4\pi}{3}-t\right)\!,\;
	\ka_2^*\!\left(\frac{4\pi}{3}-t\right)\!,\;
	-\ka_3^*\!\left(\frac{4\pi}{3}-t\right)\!\right),
\end{align*}
where
\begin{align*}
  \ka_1^*(t)
	= {} & \lim_{\la\rightarrow\pm\infty}\ka_1(\la,t) 
	= -\tan t\,,\\[0.1cm]
  \ka_2^*(t)
	= {} & \lim_{\la\rightarrow\pm\infty}\ka_2(\la,t)
	= \frac{1}{2}+\frac{1}{\cos t}\,,\\[0.1cm]
  \ka_3^*(t)
	= {} & \lim_{\la\rightarrow\pm\infty}\ka_3(\la,t) 
	= \frac{\sqrt{1+2\cos t}}{\cos t}\,.
\end{align*}
The parametrization $\g^*(t)$ of $\mathcal{C}_{\pm\infty}$ has poles for
\beq
  t_1 = -\frac{\pi}{2}\,,\quad t_2 = \frac{\pi}{2}\,,\quad
  t_3 = \frac{4\pi}{3}-\frac{\pi}{2} = \frac{5\pi}{6}\,,\quad
  t_4 = \frac{4\pi}{3}+\frac{\pi}{2} = \frac{11\pi}{6}\,,  
\eeq
and therefore, $\mathcal{C}_{\pm\infty}$ consists of four branches (see Fig.~\ref{C_infxyz}). For the asymptotes of $\mathcal{C}_{\pm\infty}$ we find 
\beqn \label{A(inf)}
\left.\begin{aligned}
  \lim_{\la\rightarrow\pm\infty}\A_1(\la)
  = {} & \left\{\left(\nu-1,\,\nu-1/2,\,\nu\right)\:|\:\nu\in\RR
		\right\},\\[0.1cm]
   \lim_{\la\rightarrow\pm\infty}\A_2(\la)
  = {} & \left\{\left(1-\nu,\,\nu-1/2,\,\nu\right)\:|\:\nu\in\RR
		\right\},\\[0.1cm]
 \lim_{\la\rightarrow\pm\infty}\A_3(\la)
  = {} & \left\{\left(1-\nu,\,\nu-1/2,\,-\nu\right)\:|\:\nu\in\RR
		\right\},\\[0.1cm] 
 \lim_{\la\rightarrow\pm\infty}\A_4(\la)
  = {} & \left\{\left(\nu-1,\,\nu-1/2,\,-\nu\right)\:|\:\nu\in\RR
		\right\},
\end{aligned}\;\;\right\}
\eeqn
and from \eqref{X_iZ_i} for the intersection points,
\begin{align*}
  \lim_{\la\rightarrow\pm\infty}X_1(\la)
  = {} & \left(0\,,\:\frac{3}{2}\,,\:\sqrt{3}\right),\hspace{0.65cm}
  \lim_{\la\rightarrow\pm\infty}Z_1(\la)
  = \left(\sqrt{3}\,,\:-\frac{3}{2}\,,\:0\right),\\[0.1cm]
  \lim_{\la\rightarrow\pm\infty}X_2(\la)
  = {} & \left(0\,,\:\frac{3}{2}\,,\:-\sqrt{3}\right),\quad
  \lim_{\la\rightarrow\pm\infty}Z_2(\la)
  = \left(-\sqrt{3}\,,\:-\frac{3}{2}\,,\:0\right).
\end{align*}
From the parametrization $\g^*(t)$ of $\mathcal{C}_{\pm\infty}$ we have
\beq
  x^2 = \tan^2 t = \frac{1-\cos^2 t}{\cos^2 t}\,,\quad
  y = \frac{1}{2}+\frac{1}{\cos t}\,,\quad
  z^2 = \frac{1+2\cos t}{\cos^2 t}\,.
\eeq
By eliminating $\cos t$ we find the following algebraic equations of the projections of $\mathcal{C}_{\pm\infty}$ onto the planes $\X$, $\Y$, $\Z$:
\begin{eqnarray}
\mbox{projection onto $\X$:} & & \left(y+\tfrac{1}{2}\right)^2-z^2=1\,,
	\label{X-hyperbola}\\[0.1cm] 
\mbox{$\Y$:} & & \left(x^2-z^2\right)^2-2\left(x^2+z^2\right)=3\,,
	\\[0.12cm]
\mbox{$\Z$:} & & \left(y-\tfrac{1}{2}\right)^2-x^2=1\,.
	\label{Z-hyperbola}
\end{eqnarray}
Formulas \eqref{X-hyperbola} and $\eqref{Z-hyperbola}$ are the equations of hyperbolas with center points $y=-1/2$, $z=0$, and $x=0$, $y=1/2$, respectively. The actual projections of $\mathcal{C}_{\pm\infty}$ onto $\X$ and $\Z$ (plotted with thick lines) are part of the respective hyperbola.\\[0.3cm]
\underline{Case 4: $\la\in\{-1,2\}$} $\C$ consists of two branches.  For $\la=-1$, from \eqref{X_iZ_i} it follows that
\begin{align*}
& \lim_{\ve\rightarrow 0}X_1(-1-\ve) = (0,\infty,\infty)\,,\hspace{0.7cm}
  \lim_{\ve\rightarrow 0}X_1(-1+\ve) = (0,-\infty,-\infty)\,,\\[0.1cm]
& \lim_{\ve\rightarrow 0}X_2(-1-\ve) =  (0,\infty,-\infty)\,,\quad
  \lim_{\ve\rightarrow 0}X_2(-1+\ve) =  (0,-\infty,\infty)\,,\\[0.1cm]
& Z_1(-1) = \left(\frac{2}{\sqrt{3}}\,,\:-\frac{1}{2}\,,\:0\right),
  \hspace{1.1cm} 
  Z_2(-1) = \left(\!-\frac{2}{\sqrt{3}}\,,\:-\frac{1}{2}\,,\:0\right),  
\end{align*}
 
\begin{figure}[h]
\begin{minipage}[c]{0.5\textwidth}
\includegraphics[scale=0.74]{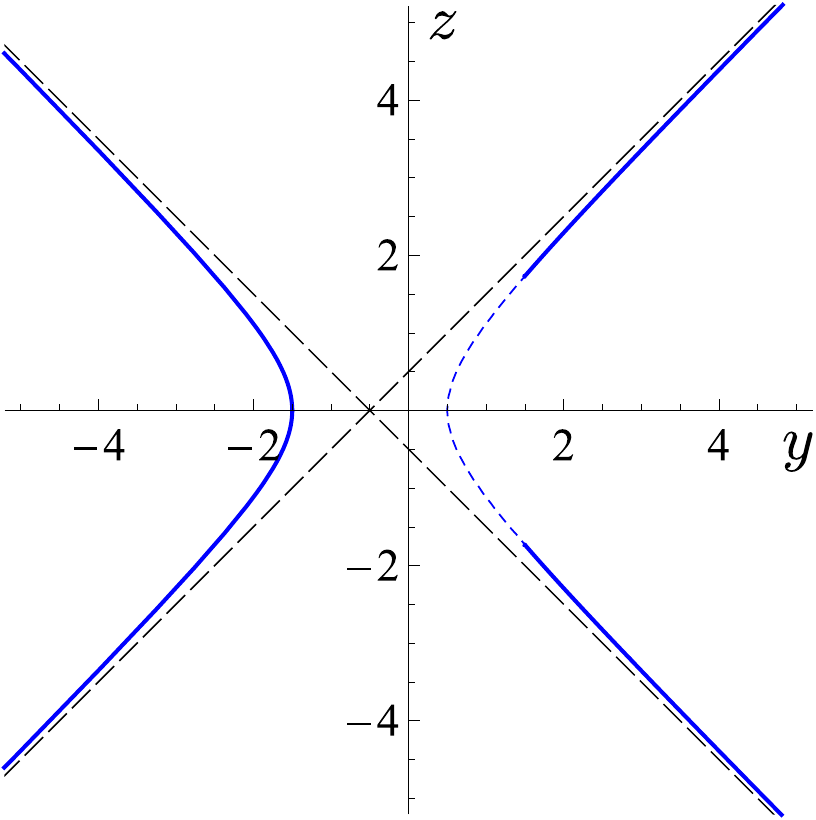}
\vspace{0.25cm}
\end{minipage}
\hspace{0.1cm}
\begin{minipage}[c]{0.5\textwidth}
\includegraphics[scale=0.74]{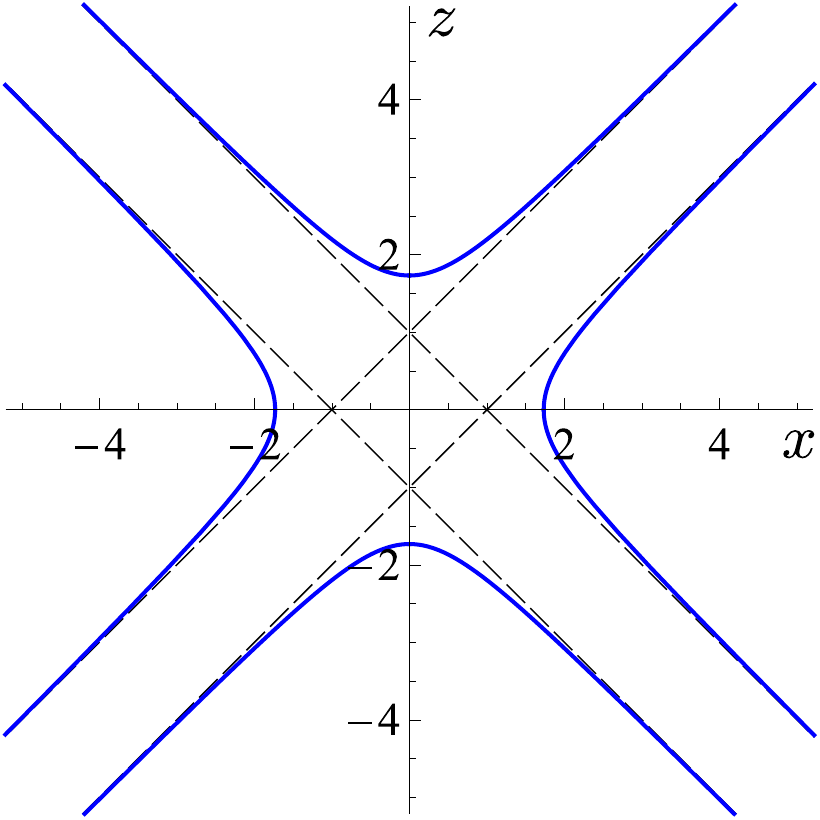}
\vspace{0.25cm}
\end{minipage}
\begin{minipage}[c]{0.5\textwidth}
\hspace{0.03cm}\includegraphics[scale=0.74]{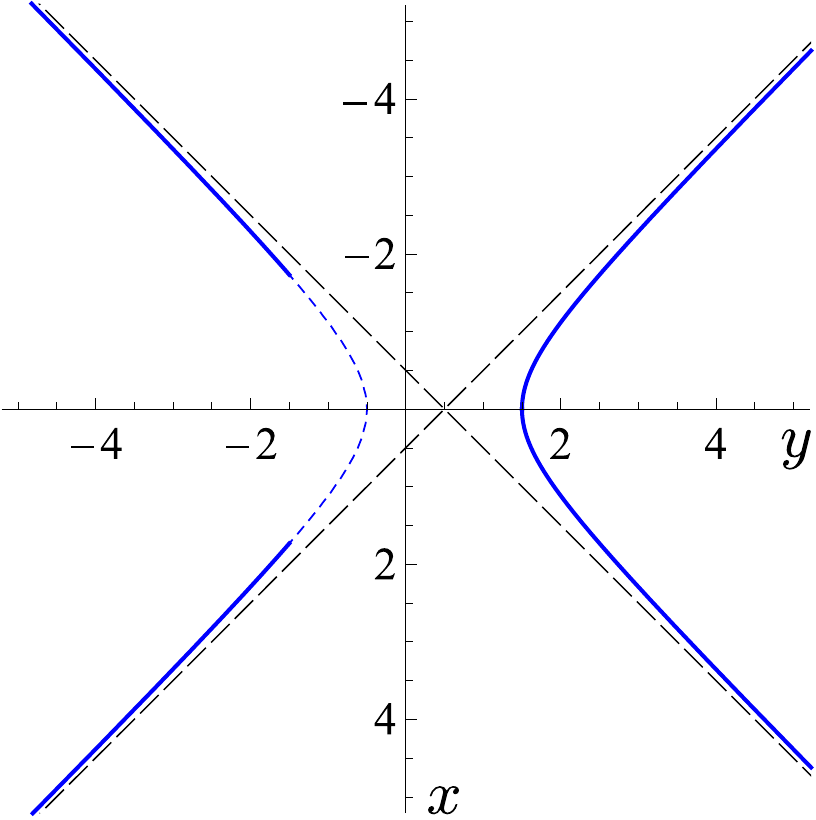}
\end{minipage}
\hspace{0.1cm}
\begin{minipage}[c]{0.5\textwidth}
\caption{\label{C_infxyz} Projections of $\mathcal{C}_{\pm\infty}$ and its asymptotes $\A_k(\pm\infty)$, $k=1,2,3,4$,$\quad$ onto the planes $\X$, $\Y$, and $\Z$;} $-5\leq x,y,z\leq 5$
\end{minipage}
\vspace{0.6cm}
\end{figure}

\begin{figure}[h]
\begin{minipage}[c]{0.5\textwidth}
\includegraphics[scale=0.74]{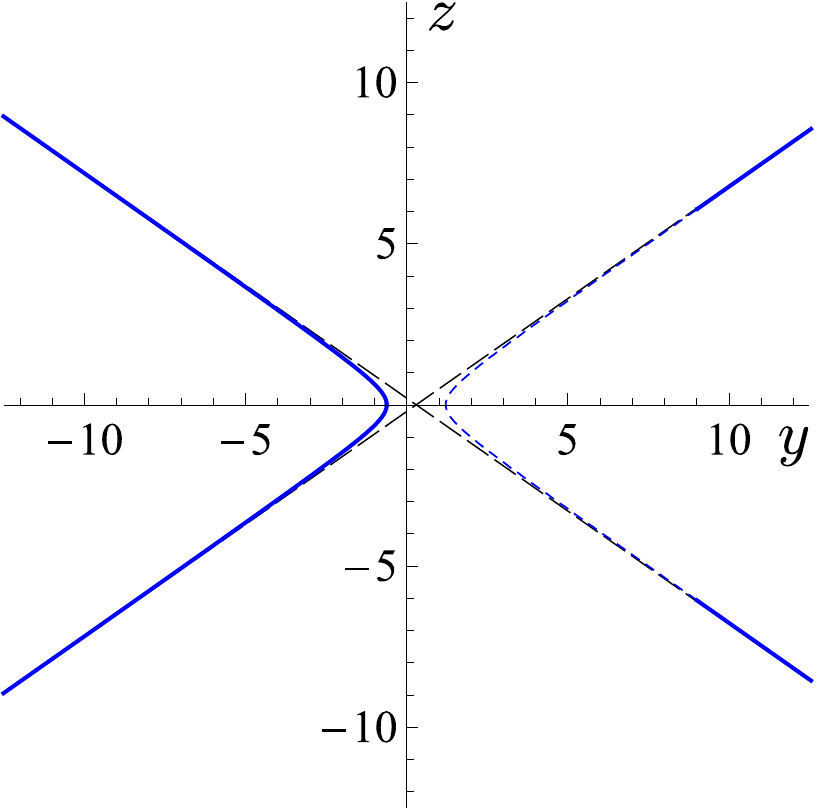}
\end{minipage}
\hspace{0.1cm}
\begin{minipage}[c]{0.5\textwidth}
\includegraphics[scale=0.74]{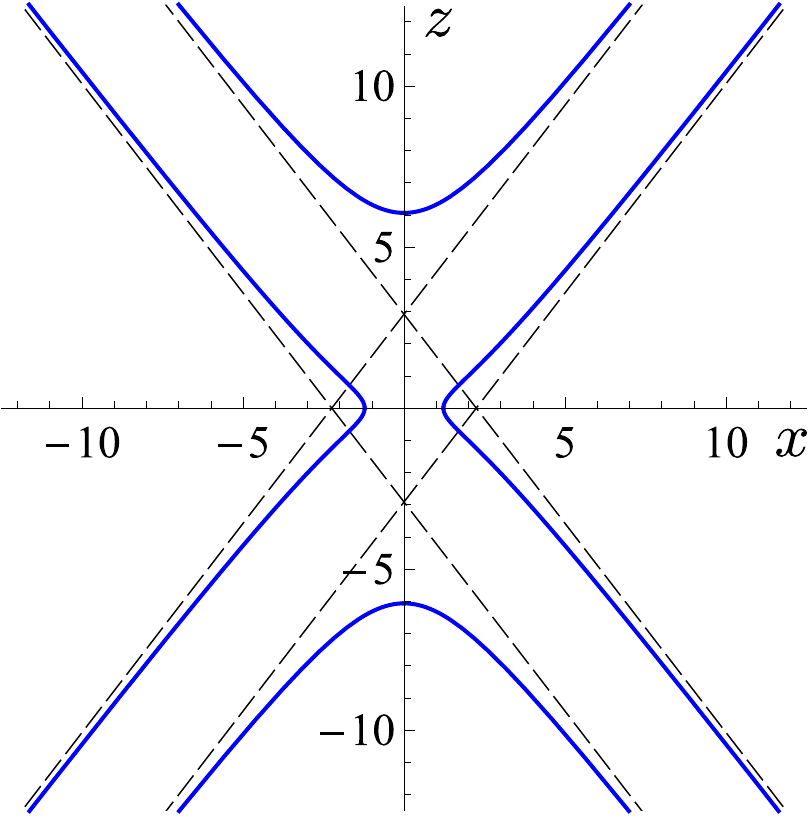}
\end{minipage}
\caption{\label{C-1.4xyz} Projections of $\mathcal{C}_{-1.4}$ and its asymptotes $\A_k(4)$, $k=1,2,3,4$, onto the planes $\X$ and $\Y$; $-12\leq x,y,z\leq 12$}
\end{figure}

\begin{figure}[h]
\begin{minipage}[c]{0.5\textwidth}
\includegraphics[scale=0.74]{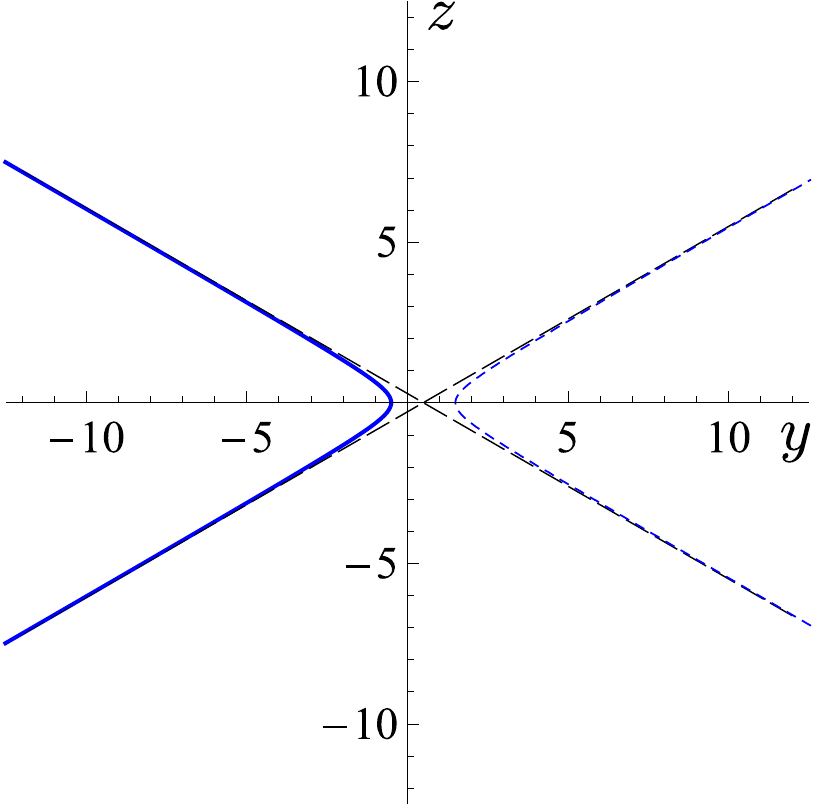}
\vspace{0.3cm}
\end{minipage}
\hspace{0.1cm}
\begin{minipage}[c]{0.5\textwidth}
\includegraphics[scale=0.74]{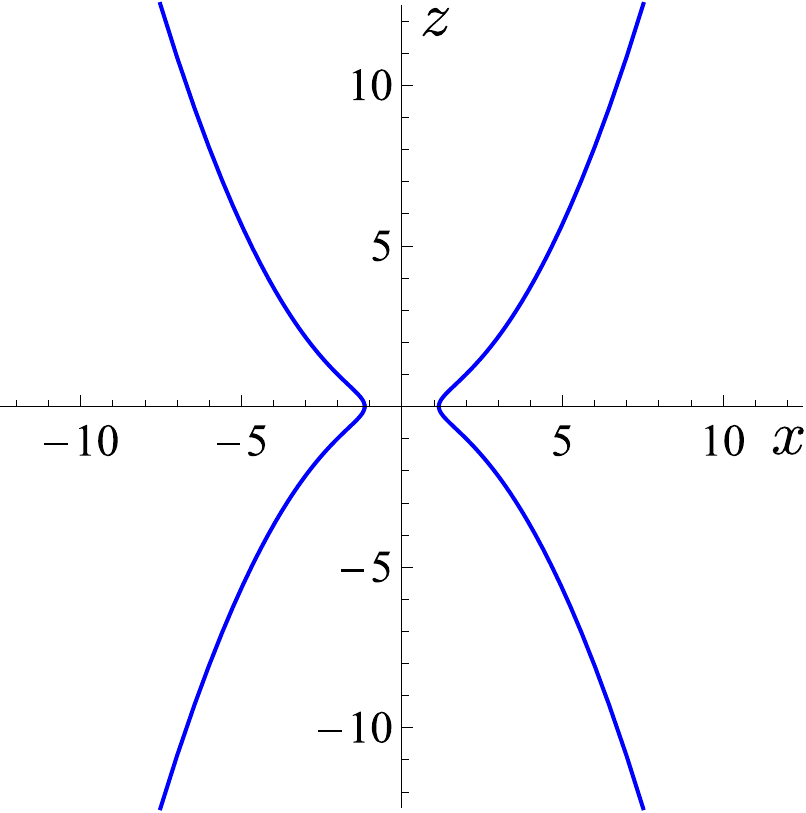}
\vspace{0.3cm}
\end{minipage}
\begin{minipage}[c]{0.5\textwidth}
\includegraphics[scale=0.74]{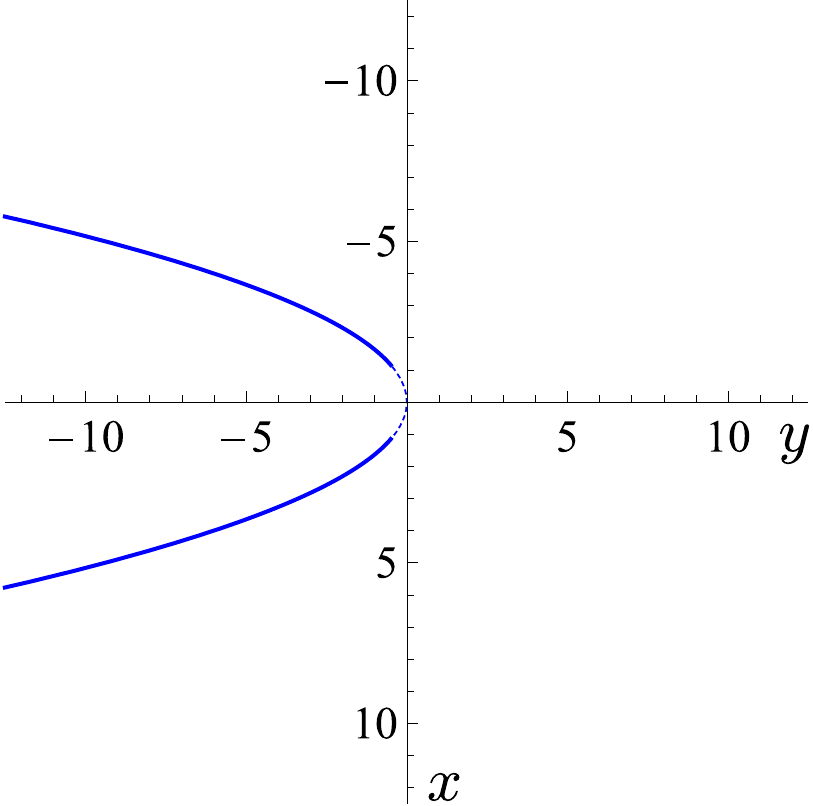}
\end{minipage}
\hspace{0.1cm}
\begin{minipage}[c]{0.5\textwidth}
\caption{\label{C-1xyz} Projections of $\mathcal{C}_{-1}$ onto the planes $\X$, $\Y$, and $\Z$;} $-12\leq x,y,z\leq 12$
\end{minipage}
\vspace{0.6cm}
\end{figure}

\begin{figure}[h]
\begin{minipage}[c]{0.5\textwidth}
\includegraphics[scale=0.74]{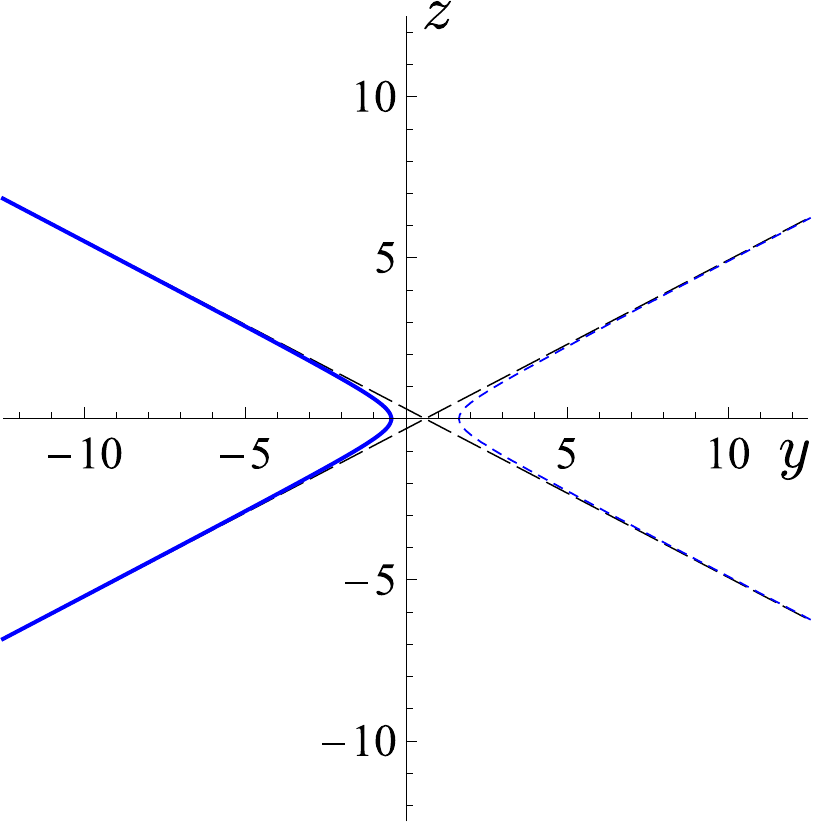}
\end{minipage}
\hspace{0.1cm}
\begin{minipage}[c]{0.5\textwidth}
\includegraphics[scale=0.74]{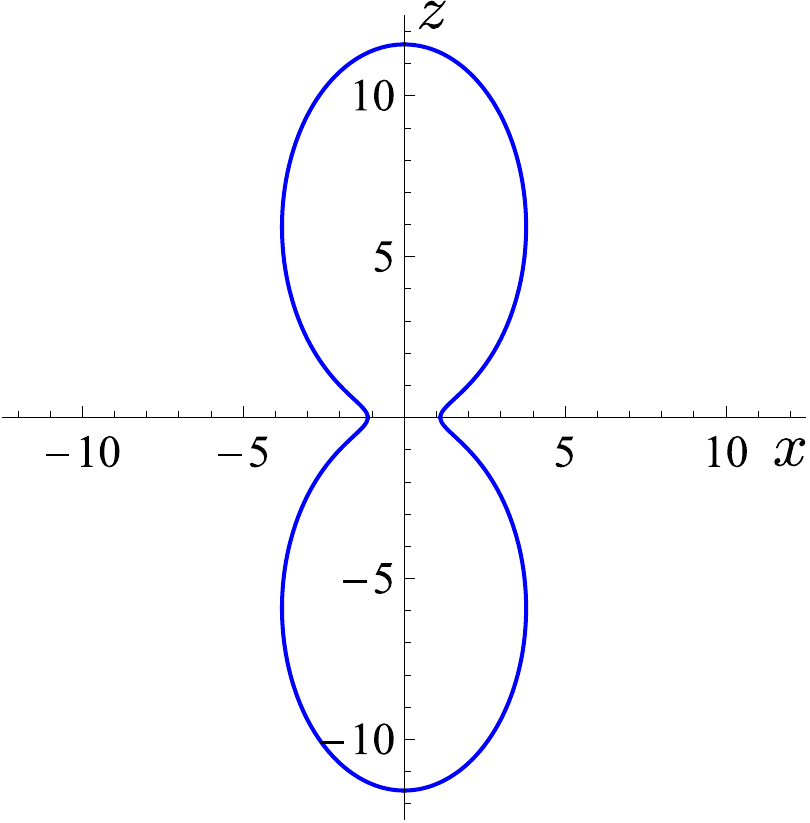}
\end{minipage}
\caption{\label{C-0.87xyz} Projections of $\mathcal{C}_{-0.87}$ onto the planes $\X$ and $\Y$; $-12\leq x,y,z\leq 12$}$\qquad\qquad\qquad\qquad\qquad$
\end{figure}

\clearpage

\noindent
and for $\la=2$,
\begin{align*}
& X_1(2) = \left(0\,,\:\frac{1}{2}\,,\:\frac{2}{\sqrt{3}}\right),
  \hspace{1.3cm}
  X_2(2) = \left(0\,,\:\frac{1}{2}\,,\:-\frac{2}{\sqrt{3}}\right),\displaybreak[0]\\[0.1cm]
& \lim_{\ve\rightarrow 0}Z_1(2-\ve) = (-\infty,\infty,0)\,,\quad
  \lim_{\ve\rightarrow 0}Z_1(2+\ve) = (\infty,-\infty,0)\,,\\
& \lim_{\ve\rightarrow 0}Z_2(2-\ve) = (\infty,\infty,0)\,,\hspace{0.67cm}
  \lim_{\ve\rightarrow 0}Z_2(2+\ve) = (-\infty,-\infty,0)\,.
\end{align*}
Projections of $\mathcal{C}_{-1}$ are shown in Fig.\ \ref{C-1xyz}. The projection onto the plane~$\X$ is the left branch (thick line) of the hyperbola with the algebraic equation
\beq
  \left(y-\frac{1}{2}\right)^2-3z^2=1\,.
\eeq
The equations of its asymptotes are $z=\pm(y-1/2)/\sqrt{3}$. The projection of $\mathcal{C}_{-1}$ onto $\Y$ has the algebraic equation
\beq
  3x^4+8x^2-64z^2 = 16\,.
\eeq
The projection onto $\Z$ is part of the parabola with equation $x=-3y^2/8$.

%% file: OOloid_3.tex
% !TeX root = OOloid_0.tex

\section{The edge of regression}

In the following we denote by $\R$ the edge of regression of the developable surface $\Ol$ (see \cite[pp.\ 119-125]{Strubecker}, and Fig.\ \ref{B01}, Fig.\ \ref{R}).
\begin{theorem} \label{Gratlinie}
For $t\in I_1$, parametric equations of $\R$ are given by
\begin{align*}
  x = g_1(t) = {} & \frac{\sin t-\tan t}{3}\,,\quad
  y = g_2(t) = \frac{2+3\cos t-3\cos^2 t-2\cos^3 t}
			{6(1+\cos t)\cos t}\,,\displaybreak[0]\\
  z = g_3(t) = {} & \!\pm\!\frac{(1+2\cos t)^{3/2}}{3(1+\cos t)\cos t}\,.
\end{align*}
\end{theorem}

\begin{proof}
$\R$ is the solution of the system of equations
\beq \label{Czuber-System}
\left. \begin{array}{l} 
  f_\la(x,y,z)\; = 0\,,\quad
  f'_\la(x,y,z):=\dfrac{\partial}{\partial\la}\,f_\la(x,y,z)=0\,,\\[0.3cm]
  f''_\la(x,y,z):= \dfrac{\partial^2}{\partial\la^2}\,f_\la(x,y,z)=0
\end{array}\right\}
\eeq
with variable $\la$, see \cite[pp. 523-524]{Czuber}. Since the touching curve $\C$ is the intersection curve of the quadric $\Q$ and the quadric defined by $f'_\la(x,y,z)=0$, the parametric functions $\ka_1$, $\ka_2$, $\ka_3$ of $\C$ are not only solutions of
\beq
  f_\la(\ka_1(\la,t),\ka_2(\la,t),\ka_3(\la,t))=0\;\;
\eeq
but also of 
\beq
  f'_\la(\ka_1(\la,t),\ka_2(\la,t),\ka_3(\la,t))=0\,.
\eeq
Furthermore, one finds
\begin{align*}
  & \hspace{-0.45cm} f''_\la(x,y,z)\\
  = {} & \frac{x^2}{(1-\la)^3} + \frac{3y^2(\la-1)\la-y\left(2-3\la-3\la^2+2\la^3\right)}
	{\left(1-\la+\la^2\right)^3} + \frac{z^2}{\la^3}
	-\frac{9(\la-1)\la}{\left(1-\la+\la^2\right)^3}\,.
\end{align*}
Solving the equation
\beq
  f''_\la(\ka_1(\la,t),\ka_2(\la,t),\ka_3(\la,t)) = 0
\eeq
for $\la$ yields
\beq
  \la = \phi(t) := \frac{1+2\cos t}{(2+\cos t)\cos t}\,.
\eeq
It follows that $g_j(t):=\ka_j(\phi(t),t)$, $j=1,2,3$, and
\beq
  x = g_1(t)\,,\quad y = g_2(t)\,,\quad z = g_3(t)\,.
\eeq
Due to the symmetry of $\Ol$ with respect to the plane $\Z$, we also have
\beq
  x = g_1(t)\,,\quad y = g_2(t)\,,\quad z = -g_3(t)\,. \qedhere
\eeq    
\end{proof}

\begin{corollary} \label{r_i}
With the system parameter $\la$, the edge of regression is given by
\begin{align*}
 \R 
  = {} &	\{r_1(\la),r_2(\la),r_3(\la)\:|\:\la\in\RR\setminus[0,1]\}\\
  \cup\hspace{0.12cm} {} & \{-r_1(\la),r_2(\la),r_3(\la)\:|\:
	\la\in\RR\setminus[0,1]\}\\
  \cup\hspace{0.12cm} {} & \{-r_1(\la),r_2(\la),-r_3(\la)\:|\:
	\la\in\RR\setminus[0,1]\}\\
  \cup\hspace{0.12cm} {} & \{r_1(\la),r_2(\la),-r_3(\la)\:|\:
	\la\in\RR\setminus[0,1]\}\,,
\end{align*}
where
\begin{align*}
  r_1(\la) = {} & \frac{\sqrt{(\la-1)^3\,[2-\la+2\rho(\la)]}}
       {\la\,[2-\la+\rho(\la)]}\,,\;\;
  r_2(\la) = \frac{\la^2+2\la-2+(\la-2)\,\rho(\la)}
       {2\la\,[2-\la+\rho(\la)]}\,,\\[0.2cm]
  r_3(\la) = {} & \frac{\mathrm{sgn}(\la)\,\sqrt{\la\,[2-\la+2\rho(\la)]}}
       {2-\la+\rho(\la)}\,,\;\;\;
  \rho(\la) = \mathrm{sgn}(\la)\,\sqrt{1-\la+\la^2}\,.
\end{align*}
\end{corollary}

\begin{proof}
We consider the function
\beq
  \phi : \;
  I_1 \;\rightarrow\; \RR\,,\quad
  t \;\mapsto\; \phi(t)=\frac{1+2\cos t}{(2+\cos t)\cos t}\,,  
\eeq
used in the proof of Theorem \ref{Gratlinie}. We denote by $\phi_1$ the restriction of $\phi$ to the interval $(-2\pi/3,0)$, and by $\phi_2$ the restriction of $\phi$ to $(0,2\pi/3)$. One easily finds the respective inverse functions
\beq
  \phi_1^{-1}(\la) = -\!\arccos\dfrac{1-\la+\rho(\la)}{\la}\,,\quad
  \phi_2^{-1}(\la) =  \arccos\dfrac{1-\la+\rho(\la)}{\la}
\eeq
with $\rho(\la):=\mathrm{sgn}(\la)\,\sqrt{1-\la+\la^2}$ and $\la\in\RR\setminus[0,1]$, hence
\begin{align*}
  r_1(\la) := {} & \ka_1\big(\la,\phi_1^{-1}(\la)\big) 
  = \frac{\sqrt{(\la-1)^3\,[2-\la+2\rho(\la)]}}
       {\la\,[2-\la+\rho(\la)]}\,,\\[0.2cm]
  r_2(\la) := {} & \ka_2\big(\la,\phi_1^{-1}(\la)\big)
  = \frac{\la^2+2\la-2+(\la-2)\,\rho(\la)}
       {2\la\,[2-\la+\rho(\la)]}\,,\\[0.2cm]
  r_3(\la) := {} & \ka_3\big(\la,\phi_1^{-1}(\la)\big)
  = \frac{\mathrm{sgn}(\la)\,\sqrt{\la\,[2-\la+2\rho(\la)]}}
       {2-\la+\rho(\la)}\,,
\end{align*}
and
\beq
  \ka_1\big(\la,\phi_2^{-1}(\la)\big) = -r_1(\la)\,,\;\;
  \ka_2\big(\la,\phi_2^{-1}(\la)\big) =  r_2(\la)\,,\;\;
  \ka_3\big(\la,\phi_2^{-1}(\la)\big) =  r_3(\la)\,.
\eeq
It follows that
\beq
  x = \pm r_1(\la)\,,\quad y = r_2(\la)\,,\quad z = r_3(\la)\,,
\eeq
and, due to the symmetry of $\Ol$ with respect to the plane $\Z$, also
\beq
  x = \pm r_1(\la)\,,\quad y = r_2(\la)\,,\quad z = -r_3(\la)\,.
  \qedhere
\eeq
\end{proof}

\noindent
Every generating line of a developable surface is a tangent to its edge of regression. Obviously, $t=-\pi/2$ and $t=\pi/2$ are poles of the functions $g_1$, $g_2$, $g_3$ in Theorem \ref{Gratlinie}. One easily finds
\beq
\begin{array}{*2{l@{\;=\;}r@{\,,\quad}}l@{\;=\;}r}
\om_1\!\left(m,-\dfrac{\pi}{2}\right) & m-1 & 
\om_2\!\left(m,-\dfrac{\pi}{2}\right) & m-\dfrac{1}{2} & 
\om_3\!\left(m,-\dfrac{\pi}{2}\right) & m\,,\\[0.3cm]
\om_1\!\left(m,\,\dfrac{\pi}{2}\right) & 1-m & 
\om_2\!\left(m,\,\dfrac{\pi}{2}\right) & m-\dfrac{1}{2} & 
\om_3\!\left(m,\,\dfrac{\pi}{2}\right) & m\,.
\end{array}
\eeq
Hence, using abbreviated notation, the four asymptotes to $\R$ are
\beqn \label{tilde(A)}
\left.\begin{aligned}
  \widetilde{\A}_1 = {} &
	\{(m-1,\,m-1/2,\,m)\:|\:m\in\RR\}\,,\\[0.1cm]
  \widetilde{\A}_2 = {} &
	\{(1-m,\,m-1/2,\,m)\:|\:m\in\RR\}\,,\\[0.1cm]
  \widetilde{\A}_3 = {} &
	\{(1-m,\,m-1/2,\,-m)\:|\:m\in\RR\}\,,\\[0.1cm]
  \widetilde{\A}_4 = {} &
	\{(m-1,\,m-1/2,\,-m)\:|\:m\in\RR\}\,.
\end{aligned}\;\;\right\}
\eeqn
If the intersection point of the asymptotes $j$ and $k$ exists, we denote it by $S_{jk}$ and find:
\beq
\begin{array}{l}
  S_{12} = \left( 0,\, \frac{1}{2},\, 1\right),\;
  S_{23} = \left( 1,\,-\frac{1}{2},\, 0\right),\;
  S_{34} = \left( 0,\, \frac{1}{2},\,-1\right),\;
  S_{41} = \left(-1,\,-\frac{1}{2},\, 0\right).
\end{array}
\eeq
Note that $S_{12},S_{34}\in k_B$, and $S_{23},S_{41}\in k_A$.

\begin{figure}[h]
\begin{minipage}[c]{0.5\textwidth}
\includegraphics[scale=0.74]{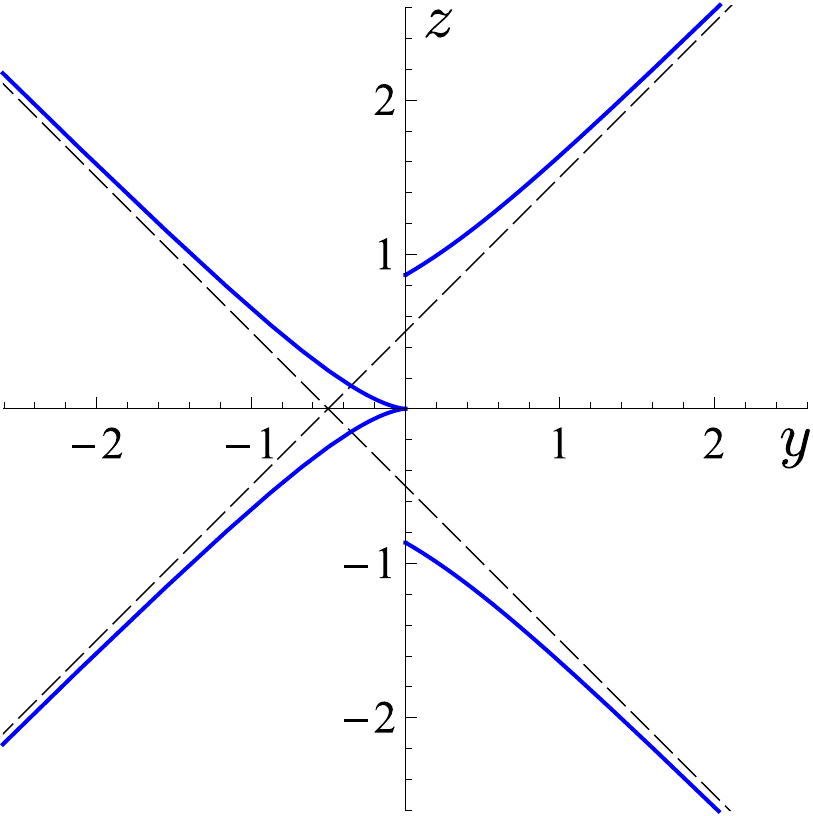}
\end{minipage}
\hspace{0.1cm}
\begin{minipage}[c]{0.5\textwidth}
\includegraphics[scale=0.74]{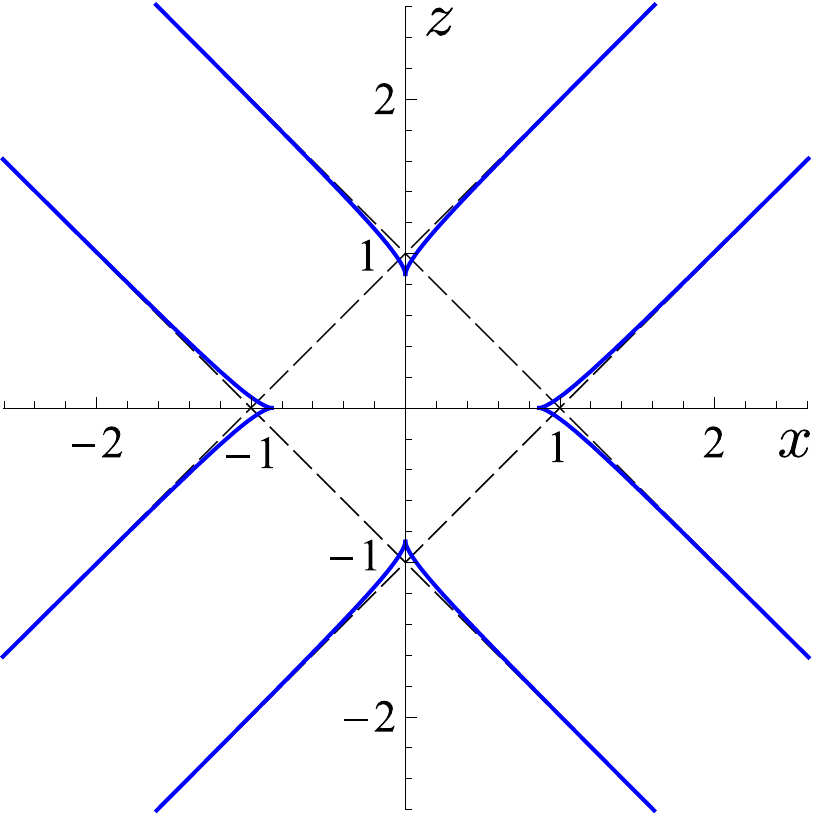}
\end{minipage}
\caption{\label{R} Projections of $\R$ and its asymptotes $\widetilde{\A}_k=\A_k(\pm\infty)$, $k=1,2,3,4$, onto the planes $\X$ and $\Y$, $-2.5\leq x,y,z\leq 2.5$}
\end{figure}

\newpage
\noindent
Comparing Eqs.\ \eqref{A(inf)} and \eqref{tilde(A)} we find:
\begin{theorem} \label{asymptotes}
For $\la\in\RR\setminus[-1,2]$ and $k\in\{1,2,3,4\}$, the asymptote $\A_k(\la)$ of $\C$ tends to the asymptote $\widetilde{\A}_k$ of $\R$ if $\la\rightarrow\pm\infty$. 
\end{theorem}

$\mathcal{C}_0$ and $\mathcal{C}_1$ are double circular arcs of $\mathcal{Q}_0=k_A$ and $\mathcal{Q}_1=k_B$, respectively, with $\mathcal{C}_0,\mathcal{C}_1\subset\Ol$ (see Fig. \ref{B02}). $\R$ has cusps in the four endpoints of the double curves $\mathcal{C}_0$, $\mathcal{C}_1$ (see \cite[p. 206]{Rohn_Papperitz}). $\mathcal{C}_0$ has the endpoints
\begin{align*}
& \left(\sin\!\left(\!-\frac{2\pi}{3}\right),\:-\frac{1}{2}
	-\cos\!\left(\!-\frac{2\pi}{3}\right),\:0\right) =  
  \left(\!-\frac{\sqrt{3}}{2}\,,\:0\,,\:0\right),\\
& \left(\sin\frac{2\pi}{3},\:-\frac{1}{2}
	-\cos\frac{2\pi}{3},\:0\right) =  
  \left(\frac{\sqrt{3}}{2}\,,\:0\,,\:0\right).
\end{align*} 
Due to the symmetry of $\Ol$, the endpoints of $\mathcal{C}_1$ are $\left(0,0,-\sqrt{3}/2\right)$, $\left(0,0,\sqrt{3}/2\right)$.

%% file: OOloid_4.tex
% !TeX root = OOloid_0.tex

\section{The self-polar tetrahedron}

\begin{theorem}
The faces of the common self-polar tetrahedron $\PT$ of the inscribed quadrics $\Q$ are formed by the planes
\begin{align*}
  \X_1
	= {} & \left\{(x_0,x_1,x_2,x_3)\in\mathbb{P}_3(\mathbb{C})\:\Big|\:
		   x_1 = 0\right\},\\[0.1cm]
  \X_3
	= {} & \left\{(x_0,x_1,x_2,x_3)\in\mathbb{P}_3(\mathbb{C})\:\Big|\:
		   x_3 = 0\right\},\displaybreak[0]\\[0.1cm]
  \mathcal{I}_1
	= {} & \left\{(x_0,x_1,x_2,x_3)\in\mathbb{P}_3(\mathbb{C})\:\Big|\:
		   x_2 = \frac{\sqrt{3}}{2}\,\ii\,x_0\right\},
		   \displaybreak[0]\\[0.1cm]
  \mathcal{I}_2
	= {} & \left\{(x_0,x_1,x_2,x_3)\in\mathbb{P}_3(\mathbb{C})\:\Big|\:
		   x_2 = -\frac{\sqrt{3}}{2}\,\ii\,x_0\right\}.
\end{align*}
The vertices of $\PT$ are the ideal points $X_\infty=[0,1,0,0]$ and $Z_\infty=[0,0,0,1]$ of the $x$-axis and the $z$-axis, respectively, and the points
\beq
  P = \left[1,\,0,\,\frac{\sqrt{3}}{2}\,\ii,\,0\right]\,,\quad
  Q = \left[1,\,0,\,-\frac{\sqrt{3}}{2}\,\ii,\,0\right]\,.
\eeq
\end{theorem}

\begin{proof}
In a tangential system of quadrics there are four which degenerate to conics, and their planes are the faces of the common self-polar tetrahedron \cite[pp.\ 205]{Rohn_Papperitz}, \cite[p.\ 254]{Sommerville}. $\Q$ degenerates to a conic if one of the denominators in
$\tilde{f}_\la(x_0,x_1,x_2,x_3)$, see \eqref{tilde(f)}, vanishes.

For $\lambda=0$ it follows that $x_3^2=0$, and therefore the plane $\X_3$ delivers the first face of the tetrahedron $\PT$. For $\lambda=1$ we have  $x_1^2=0$, and hence $\X_1$ is the second face of $\PT$. The two remaining faces follow from the equation $1-\lambda+\lambda^2=0$. Its solutions are
\beq
  \lambda_1 = \frac{1}{2}+\frac{\sqrt{3}}{2}\,\ii
  \quad\mbox{and}\quad
  \lambda_2 = \frac{1}{2}-\frac{\sqrt{3}}{2}\,\ii\,.
\eeq
So we have
\begin{align*}
  \left(x_2+\left(\frac{1}{2}-\la_1\right)x_0\right)^2
  =   {} & \left(x_2-\frac{\sqrt{3}}{2}\,\ii\,x_0\right)^2 = 0\,,\\
  \left(x_2+\left(\frac{1}{2}-\la_2\right)x_0\right)^2
  =   {} & \left(x_2+\frac{\sqrt{3}}{2}\,\ii\,x_0\right)^2 = 0\,; 
\end{align*}
hence the planes $\mathcal{I}_1$ and $\mathcal{I}_2$ are the third and fourth face of $\PT$. The vertices of $\PT$ are the intersection points of the respective planes:
\begin{align*}
  \X_1\cap\X_3\cap\mathcal{I}_1 = {} & P\,,\hspace{0.7cm}
  \X_1\cap\X_3\cap\mathcal{I}_2 = Q\,,\\[0.1cm]
  \X_3\cap\,\mathcal{I}_1\cap\mathcal{I}_2 = {} & X_\infty\,,\quad
  \X_1\cap\,\mathcal{I}_1\cap\mathcal{I}_2 = Z_\infty\,.\qedhere
\end{align*}   
\end{proof}

\noindent
The degenerate quadrics $\mathcal{Q}_0$ and $\mathcal{Q}_1$ are the circles $k_A$ and $k_B$, respectively. For the degenerate quadrics $\mathcal{Q}_{\lambda_1}$, $\mathcal{Q}_{\lambda_2}$ we find the equations
\begin{align*}
 f_{\lambda_1}(x,y,z) = {} &
  \frac{x^2}{\left(\frac{1}{2}-\frac{\sqrt{3}}{2}\,\ii\right)^2} +
  \frac{z^2}{\left(\frac{1}{2}+\frac{\sqrt{3}}{2}\,\ii\right)^2} - 1
  = 0\,,\displaybreak[0]\\
 f_{\lambda_2}(x,y,z) = {} &
  \frac{x^2}{\left(\frac{1}{2}+\frac{\sqrt{3}}{2}\,\ii\right)^2} +
  \frac{z^2}{\left(\frac{1}{2}-\frac{\sqrt{3}}{2}\,\ii\right)^2} - 1
  = 0\,.
\end{align*} 
$\mathcal{Q}_{\lambda_1}$ and $\mathcal{Q}_{\lambda_2}$ can be considered as ellipses with complex lenghts
\beq
  \frac{1}{2}+\frac{\sqrt{3}}{2}\,\ii \quad\mbox{and}\quad
  \frac{1}{2}-\frac{\sqrt{3}}{2}\,\ii
\eeq
of their semi-axes, and respective center points
\beq
  \left(0,\,\frac{\sqrt{3}}{2}\,\ii,\,0\right) \quad\mbox{and}\quad
  \left(0,\,-\frac{\sqrt{3}}{2}\,\ii,\,0\right).
\eeq
Note that these center points are the vertices $P$ and $Q$ of $\PT$ in non homogeneous coordinates.  

%% file: OOloid_5.tex
% !TeX root = OOloid_0.tex

\section{Common generating lines of $\Q$ and $\Ol$}

Only two of the conic sections $\mathcal{Q}_0$, $\mathcal{Q}_1$, $\mathcal{Q}_{\la_1}$, $\mathcal{Q}_{\la_2}$ are real. Therefore (see \cite[p.\ 206]{Rohn_Papperitz}), the quadrics $\Q$ are divided into two sets; one of these sets consists of ruled surfaces, each having four common generating lines with the developable surface $\Ol$. Clearly, these ruled surfaces are the one-sheeted hyperboloids $\Q$, $\la\in\RR\setminus[0,1]$, and the hyperbolic paraboloid $\mathcal{Q}_{\pm\infty}$.

\begin{theorem} \label{common_lines}
{\em (i)} For fixed value of $\la\in\RR\setminus[0,1]$, the four common generating lines of the one-sheeted hyperboloid $\Q$ and the extended oloid $\Ol$ are
\begin{align*}
 \G_1(\la) = {} & \{(\ot_1(m,\la),\ot_2(m,\la),\ot_3(m,\la))\:|\:
	m\in\RR\}\,,\\[0.1cm]
 \G_2(\la) = {} & \{(-\ot_1(m,\la),\ot_2(m,\la),\ot_3(m,\la))\:|\:
	m\in\RR\}\,,\\[0.1cm]
 \G_3(\la)= {} & \{(-\ot_1(m,\la),\ot_2(m,\la),-\ot_3(m,\la))\:|\:
	m\in\RR\}\,,\\[0.1cm]
 \G_4(\la) = {} & \{(\ot_1(m,\la),\ot_2(m,\la),-\ot_3(m,\la))\:|\:
	m\in\RR\}\,,
\end{align*}
with
\begin{align*}
  \ot_1(m,\la) = {} & (1-m)\,\frac{\sqrt{(\la-1)(2-\la+2\rho(\la))}}
	{-|\la|}\,,\\[0.1cm]
  \ot_2(m,\la) = {} & (1-m)\,\frac{\la-2-2\rho(\la)}{2\la}
	+ m\,\frac{2\la-1-\rho(\la)}{2(1+\rho(\la))}\,,\\[0.1cm]
  \ot_3(m,\la) = {} & m\,\frac{\mathrm{sgn}(\la)\,
	\sqrt{\la(2-\la+2\rho(\la))}}{1+\rho(\la)}\,,
\end{align*}
where $\rho(\la)=\mathrm{sgn}(\la)\,\sqrt{1-\la+\la^2}$.\\[0.2cm]
{\em (ii)} $\G_1(\la)$, $\G_2(\la)$, $\G_3(\la)$, $\G_4(\la)$ are the tangents to $\R$, and to $\C$, in the respective points
\beq
 \begin{array}{ll}
   P_1(\la) = (r_1(\la),r_2(\la),r_3(\la))\,, &
   P_2(\la) = (-r_1(\la),r_2(\la),r_3(\la))\,,\\[0.2cm]
   P_3(\la) = (-r_1(\la),r_2(\la),-r_3(\la))\,, &
   P_4(\la) = (r_1(\la),r_2(\la),-r_3(\la))
 \end{array}
\eeq
with $r_j(\la)$, $j=1,2,3$, according to Corollary \ref{r_i}.    
\end{theorem}

\begin{proof}
(i) As already known, the parametric equations of the generating lines of $\Ol$ are
\begin{align*}
  x = \hspace{0.28cm}\om_1(m,t) = {} & (1-m)\sin t,\displaybreak[0]\\[0.1cm]
  y = \hspace{0.28cm}\om_2(m,t) = {} & (1-m)
	\left(-\frac{1}{2}-\cos t\right)
	+ m\left(\frac{1}{2}-\frac{\cos t}{1+\cos t}\right),\\[0.1cm]
  z = \pm\om_3(m,t) = {} &\frac{m\,\sqrt{1+2\cos t}}{1+\cos t}\,.
\end{align*}
Substituting $x=\om_1(m,t)$, $y=\om_2(m,t)$, $z=\om_3(m,t)$ in $f_\la(x,y,z)=0$, and solving this equation for~$t$, we find
\beq
  t = \pm\tilde{t}(\la) \quad\mbox{with}\quad
  \tilde{t}(\la) = \arccos\frac{1-\la\pm\sqrt{1-\la+\la^2}}{\la}\,.
\eeq
Since we are only interested in real solutions, we can write
\beqn \label{tilde{t}}
  \tilde{t}(\la)
	= \arccos\frac{1-\la+\rho(\la)}{\la}
\eeqn
with the function $\rho$ from Corollary \ref{r_i}.
It follows that
\begin{align*}
 \om_1(m,\pm\tilde{t}(\la))
  = {} & \!\pm\!(1-m)\,\frac{\sqrt{(\la-1)(2-\la+2\rho(\la))}}
	{|\la|}\,,\\[0.1cm]
 \om_2(m,\pm\tilde{t}(\la))
  = {} & (1-m)\,\frac{\la-2-2\rho(\la)}{2\la}
	+ m\,\frac{2\la-1-\rho(\la)}{2(1+\rho(\la))}\,,\\[0.1cm]
 \om_3(m,\pm\tilde{t}(\la))
  = {} & m\,\frac{\mathrm{sgn}(\la)\,\sqrt{\la(2-\la+2\rho(\la))}}
	{1+\rho(\la)}\,.
\end{align*}
We put $\ot_j(m,\la):=\om_j(m,-\tilde{t}(\la))$, $j=1,2,3$. This yields $\G_1(\la)$ and $\G_2(\la)$. Due to the symmetry of $\Q$ and $\Ol$ with respect to the plane $\Z$, the lines $\G_3(\la)$ and $\G_4(\la)$ follow.\\[0.2cm]
(ii) The tangent
\beq
  T_\la(t) = \{(\tau_1(\la,t,\mu),\,\tau_2(\la,t,\mu),\,\tau_3(\la,t,\mu))
	\:|\:\mu\in\RR\}\,,\quad t\in I_1\,,
\eeq
to $\C$, see \eqref{T_la}, is a generating line of $\Q$ for all values of $t$ that are solutions of
\beq
  f_\la(\tau_1(\la,t,\mu),\tau_2(\la,t,\mu),\tau_3(\la,t,\mu)) = 0\,.
\eeq
One finds $t=\pm\tilde{t}(\la)$ with $\tilde{t}(\la)$ from \eqref{tilde{t}}. At first we consider only $t=-\tilde{t}(\la)$. Calculation shows that
\beq
  \ka_j(\la,-\tilde{t}(\la)) = r_j(\la)\,,\quad j=1,2,3.
\eeq
Hence, the tangent $T^{(1)}(\la):=T_\la(-\tilde{t}(\la))$ touches $\C$ in the point $P_1(\la)=(r_1(\la),r_2(\la),r_3(\la))\in\R$. According to \cite[p.\ 489]{Czuber}, $T^{(1)}(\la)$ is equal to the tangent to $\R$ in this point. The common generating lines are tangents to the edge of regression \cite[p.\ 206]{Rohn_Papperitz}. Thus one finds  
\beq
  \ot_j(\widehat{m}(\la),\la) = r_j(\la)\,,\quad j=1,2,3,
\eeq
for
\beq
  \widehat{m}(\la) := \frac{1+\rho(\la)}{2-\la+\rho(\la)}\,.
\eeq
It follows that $T^{(1)}(\la)=\G_1(\la)$. Due to symmetry with respect to the planes $\X$ and $\Z$, with $\tilde{\tau}_j(\la,\mu):=\tau_j(\la,-\tilde{t}(\la),\mu)$ we also have   
\begin{align*}
  T^{(2)}(\la)
  := {} & \{(-\tilde{\tau}_1(\la,\mu),\,\tilde{\tau}_2(\la,\mu),\,
	\tilde{\tau}_3(\la,\mu))
	\:|\:\mu\in\RR\} \hspace{0.29cm} = \G_2(\la)\,,\\
  T^{(3)}(\la)
  := {} & \{(-\tilde{\tau}_1(\la,\mu),\,\tilde{\tau}_2(\la,\mu),\,
	-\tilde{\tau}_3(\la,\mu))
	\:|\:\mu\in\RR\} = \G_3(\la)\,,\\
  T^{(4)}(\la)
  := {} & \{(\tilde{\tau}_1(\la,\mu),\,\tilde{\tau}_2(\la,\mu),\,
	-\tilde{\tau}_3(\la,\mu))
	\:|\:\mu\in\RR\} \hspace{0.29cm} = \G_4(\la)\,. \qedhere
\end{align*}
\end{proof}

\noindent
As an example, Fig.\ \ref{B03} shows the common generating lines $\G_1(4)$, $\G_2(4)$, $\G_3(4)$, $\G_4(4)$ of $\mathcal{Q}_4$ and $\Ol$.

\begin{corollary}
The four common generating lines of the hyperbolic paraboloid $\mathcal{Q}_{\pm\infty}$ and the extended oloid $\Ol$ are equal to the common asymptotes of the edge of regression $\R$ and the touching curve $\mathcal{C}_{\pm\infty}$.
\end{corollary}

\begin{proof}
From Theorem \ref{common_lines} one finds
\beq
  \lim_{\la\rightarrow\pm\infty}\ot_1(m,\la) = m-1\,,\;\;
  \lim_{\la\rightarrow\pm\infty}\ot_2(m,\la) = m-\frac{1}{2}\,,\;\;
  \lim_{\la\rightarrow\pm\infty}\ot_3(m,\la) = m\,,
\eeq
thus
\begin{align*}
 \G_1(\pm\infty) = {} & \{(m-1,\,m-1/2,\,m)\:|\:
	m\in\RR\}\,,\\[0.1cm]
 \G_2(\pm\infty) = {} & \{(1-m,\,m-1/2,\,m)\:|\:
	m\in\RR\}\,,\\[0.1cm]
 \G_3(\pm\infty) = {} & \{(1-m,\,m-1/2,\,-m)\:|\:
	m\in\RR\}\,,\\[0.1cm]
 \G_4(\pm\infty) = {} & \{(m-1,\,m-1/2,\,-m)\:|\:
	m\in\RR\}\,.
\end{align*}
Now, the result follows from Theorem \ref{asymptotes} with Eqs.\ \eqref{tilde(A)}. 
\end{proof}

%% file: OOloid_6.tex
% !TeX root = OOloid_0.tex

\section{The development of $\Ol$}

Now we consider the development of the extended oloid $\Ol$ onto its tangent plane $E$. For this, we define a cartesian $\xi,\eta$-coordinate system in $E$ as follows: Let $E$ touch $\Ol$ along the generating line
\beq
  \mathcal{L}_0= \left\{\om_1(m,0),\,\om_2(m,0),\,-\om_3(m,0)\:|\:m\in\RR\right\}
\eeq
(see \eqref{omega_i} and the proof of Corollary \ref{touching_curve}). Then $\mathcal{L}_0$ is the $\eta$-axis, and the line perpendicular to $\mathcal{L}_0$ in the point  $m=0$ is the $\xi$-axis.

\noindent
Any curve $\mathcal{C}\subset\Ol$ is developed onto a plane curve~$\mathcal{C}^*\subset E$. A parametrization of $\mathcal{C}^*$ with the arc length $t$ of the double circular arc $\mathcal{C}_0$ as parameter can be obtained from the vector transformation in \cite[p.\ 114, Theorem 4]{Dirnboeck_Stachel}. In the following for abbreviation we put $c=\cos t$, $s=\sin t$. $\lfloor\cdot\rfloor$ denotes the integer part of~$\cdot$\,.     
\begin{theorem} \label{C*}
The development of the touching curve~$\C$ onto $E$ is the curve 
\beq
  \C^* = \left\{(\ka_1^*(\la,t),\,\ka_2^*(\la,t))\:|\:t\in\RR\right\}
\eeq
with parametrization
\begin{align*}
  \ka_1^*(\la,t) = {} & 
  \mathrm{sgn}(t)\cdot\left\lfloor\frac{3\,|t|}{4\pi}+\frac{1}{2}\right\rfloor\cdot\frac{4\pi}{3\,\sqrt{3}}
  + \mathrm{sgn}(h(t))\cdot\tilde{\ka}_1(\la,h(t))\,,\\[0.1cm]
  \ka_2^*(\la,t) = {} & 
  \tilde{\ka}_2(\la,h(t))\,,
\end{align*} 
where
\begin{align*}
  \tilde{\ka}_1(\la,t) = {} & 
	\frac{2\,\sqrt{3}}{9}\left(\arccos\frac{\sqrt{2}\,c}{\sqrt{1+c}}
	+ \frac{(1-2\la)\left|s\right|\sqrt{2(1+2c)}}{(1+\la c)\,\sqrt{1+c}}\right),
	\\[0.1cm]
  \tilde{\ka}_2(\la,t) = {} & 
	\frac{\sqrt{3}}{9}\left(\ln\frac{2}{1+c}
	+ \frac{4+7\la+(11\la-4)c}{1+\la c}\right),\\[0.1cm]
  h(t) = {} & t-\mathrm{sgn}(t)\cdot\left\lfloor\frac{3\,|t|}{4\pi}+\frac{1}{2}
	\right\rfloor\cdot\frac{4\pi}{3}\,.
\end{align*} 

\end{theorem}

\begin{proof}
Substituting $x=\ka_1(\la,t)$, $y=\ka_2(\la,t)$, $z=-\ka_3(\la,t)$, see Corollary~\ref{touching_curve}, in the vector transformation \cite[p.~114, Theorem 4]{Dirnboeck_Stachel}, a straight-forward calculation delivers 
\begin{align*}
  \xi = \tilde{\tilde{\ka}}_1(\la,t) = {} & 
	\frac{2\,\sqrt{3}}{9}\left(\arccos\frac{\sqrt{2}\,c}{\sqrt{1+c}}
	+ \frac{(1-2\la)\,s\,\sqrt{2(1+2c)}}{(1+\la c)\,\sqrt{1+c}}\right),
	\\[0.1cm]
  \eta = \tilde{\ka}_2(\la,t) = {} & 
	\frac{\sqrt{3}}{9}\left(\ln\frac{2}{1+c}
	+ \frac{4+7\la+(11\la-4)c}{1+\la c}\right).
\end{align*} 
$\tilde{\kappa}_2(\la,t)$ is valid for $t\in I_1$ (see \eqref{I_i}). The periodic continuation of $\tilde{\kappa}_2(\la,t)$ yields $\ka_2^*(\la,t)$, valid for $t\in\RR$.

$\tilde{\tilde{\ka}}_1(\la,t)$ is valid only for $t\in[0,2\pi/3]$. The restriction of $\ka_1^*(\la,t)$ to $t\in I_1$ must be an odd function. Replacing $\sin t$ by $\left|\sin t\right|$ in $\tilde{\tilde{\ka}}_1(\la,t)$, we get the even function $\tilde{\ka}_1(\la,t)$, $t\in I_1$. Now, $\ka_1^\diamond(\la,t):=\mathrm{sgn}(t)\cdot\tilde{\ka}_1(\la,t)$ is the required restriction of $\ka_1^*(\la,t)$. We get
\beq
  \ka_1^\diamond(\la,2\pi/3)-\ka_1^\diamond(\la,-2\pi/3)=\frac{4\pi}{3\,\sqrt{3}}\,,
\eeq
and therefore, using the step function
\beq
  \mathrm{sgn}(t)\cdot\left\lfloor\frac{3\,|t|}{4\pi}+\frac{1}{2}\right\rfloor\cdot\frac{4\pi}{3\,\sqrt{3}}\,,
\eeq
we have found $\ka_1^*(\la,t)$, valid for $t\in\RR$.      
\end{proof}

\noindent
For $\la=0$ we have 
\begin{align*}
  \tilde{\ka}_1(0,t) = {} &
  \frac{2\,\sqrt{3}}{9}\left(\arccos\frac{\sqrt{2}\,c}{\sqrt{1+c}}
	+ \frac{\left|s\right|\sqrt{2(1+2c)}}{\sqrt{1+c}}\right),\\[0.1cm]
  \tilde{\ka}_2(0,t) = {} & 
	\frac{\sqrt{3}}{9}\left(\ln\frac{2}{1+c}	+ 4(1-c)\right).  
\end{align*}
The following manipulation of the second term in the brackets of $\tilde{\ka}_1(0,t)$,
\begin{align*}
  \frac{\left|\sin t\right|\sqrt{2(1+2\cos t)}}{\sqrt{1+\cos t}}
  = \frac{2\left|\sin\frac{t}{2}\right|\left|\cos\frac{t}{2}\right|\sqrt{2(1+2\cos t)}}
		{\sqrt{2\cos^2\frac{t}{2}}}\displaybreak[0]\\
  & \hspace{-8.4cm} = \sqrt{2}\:\left|\sin\frac{t}{2}\right|\sqrt{2(1+2\cos t)}
  = \sqrt{2}\;\sqrt{\frac{1}{2}(1-\cos t)}\;\sqrt{2(1+2\cos t)}\\[0.2cm]
  & \hspace{-8.4cm} = \sqrt{2(1+2\cos t)(1-\cos t)}\,,
\end{align*}
shows that we have obtained the result of \cite[p.\ 108, Theorem~2]{Dirnboeck_Stachel}. For $\la=1$ we get
\begin{align*}
  \tilde{\ka}_1(1,t) = {} &
  \frac{2\,\sqrt{3}}{9}\left(\arccos\frac{\sqrt{2}\,c}{\sqrt{1+c}}
	- \frac{\left|s\right|\sqrt{2(1+2c)}}{(1+c)^{3/2}}\right),\\[0.1cm]
  \tilde{\ka}_2(1,t) = {} & 
	\frac{\sqrt{3}}{9}\left(\ln\frac{2}{1+c}	+ \frac{11+7c}{1+c}\right).  
\end{align*}
This is the result of \cite[p.\ 112, Theorem~3]{Dirnboeck_Stachel}. For $\la=1/2$ we have
\begin{align*}
  \tilde{\ka}_1(1/2,t) = {} &
  \frac{2\,\sqrt{3}}{9}\arccos\frac{\sqrt{2}\,c}{\sqrt{1+c}}\,,\\[0.1cm]
  \tilde{\ka}_2(1/2,t) = {} & 
	\frac{\sqrt{3}}{9}\left(\ln\frac{2}{1+c}	+ \frac{3(5+c)}{2+c}\right),  
\end{align*}
which is the result of \cite[p.\ 115]{Dirnboeck_Stachel}. Examples with $\la\in[0,1]$ are shown in Fig.~\ref{Abw_0-1}.

\begin{figure}[h]
\begin{center}
  \includegraphics[scale=0.78]{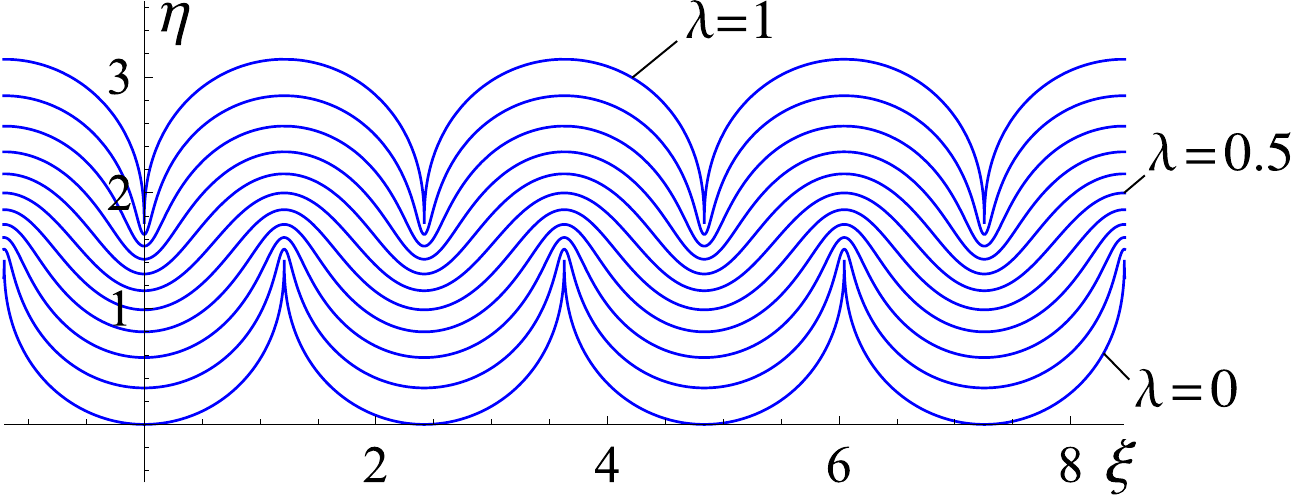}
  \caption{\label{Abw_0-1} The curves $\mathcal{C}_\la^*$, $\la=0,\,0.1,\,0.2,\ldots,\,0.9,\,1$}
\end{center}
\vspace{-0.8cm}
\end{figure}

\noindent
For $\la\rightarrow\pm\infty$ (see Fig.\ \ref{Abw_inf_a} and Fig.\ \ref{Abw_inf}) one easily finds
\begin{align*}
  \lim_{\la\rightarrow\pm\infty}\tilde{\kappa}_1(\la,t) = {} & 
	\frac{2\,\sqrt{3}}{9}\left(\arccos\frac{\sqrt{2}\,c}{\sqrt{1+c}}
	- \frac{2\left|s\right|\sqrt{2(1+2c)}}{c\,\sqrt{1+c}}\right),\\[0.1cm] 
  \lim_{\la\rightarrow\pm\infty}\tilde{\kappa}_2(\la,t) = {} &
	\frac{\sqrt{3}}{9}\left(\ln\frac{2}{\sqrt{1+c}}
	+ \frac{7+11c}{c}\right). 
\end{align*}
As further examples, the curves $\mathcal{C}_{-0.5}^*$ and $\mathcal{C}_{1.5}^*$ are shown in Fig.\ \ref{Abw_1_5__-0_5_a} and Fig.~\ref{Abw_1_5__-0_5}.\\[0.2cm] 
After substituting 
\begin{align*}
  x = g_1(t) = {} & \frac{\sin t-\tan t}{3}\,,\quad
  y = g_2(t) = \frac{2+3\cos t-3\cos^2 t-2\cos^3 t}
			{6(1+\cos t)\cos t}\,,\displaybreak[0]\\
  z = g_3(t) = {} & -\frac{(1+2\cos t)^{3/2}}{3(1+\cos t)\cos t}
\end{align*}
(see Theorem \ref{Gratlinie}) in the vector transformation \cite[p.~114, Theorem 4]{Dirnboeck_Stachel}, analogous steps as in the proof of Theorem \ref{C*} result in the development $\R^*$ of the edge of regression $\R$. We state the result in the following theorem.

\begin{theorem}
The development of $\R$ onto $E$ is the curve 
\beq
  \R^* = \left\{(g_1^*(t),\,g_2^*(t))\:|\:t\in\RR\right\}
\eeq
with parametrization
\begin{align*}
  g_1^*(t) = {} & 
  \mathrm{sgn}(t)\cdot\left\lfloor\frac{3\,|t|}{4\pi}+\frac{1}{2}\right\rfloor\cdot\frac{4\pi}{3\,\sqrt{3}}
  + \mathrm{sgn}(h(t))\cdot \tilde{g}_1(h(t))\,,\\[0.1cm]
  g_2^*(t) = {} & 
  \tilde{g}_2(h(t))\,,
\end{align*} 
where
\begin{align*}
  \tilde{g}_1(t) = {} & 
	\frac{2\,\sqrt{3}}{9}\left(\arccos\frac{\sqrt{2}\,c}{\sqrt{1+c}}
	- \frac{(2+2c-c^2)\,\sqrt{2(1+2c)}\left|s\right|}{3c\,(1+c)^{3/2}}\right),
	\\[0.1cm]
  \tilde{g}_2(t) = {} & 
	\frac{\sqrt{3}}{9}\left(\ln\frac{2}{1+c}
	+ \frac{7+33c+18c^2-4c^3}{3c\,(1+c)}\right).    
\end{align*} 
\end{theorem}
\vspace{0.2cm}
\noindent
$\R^*$ is shown in the Figures \ref{Abw_1_5__-0_5_a}, \ref{Abw_inf_a}, \ref{Abw_inf}. 

\begin{figure}[h]
\begin{center}
  \hspace{0.06cm}\includegraphics[scale=0.93]{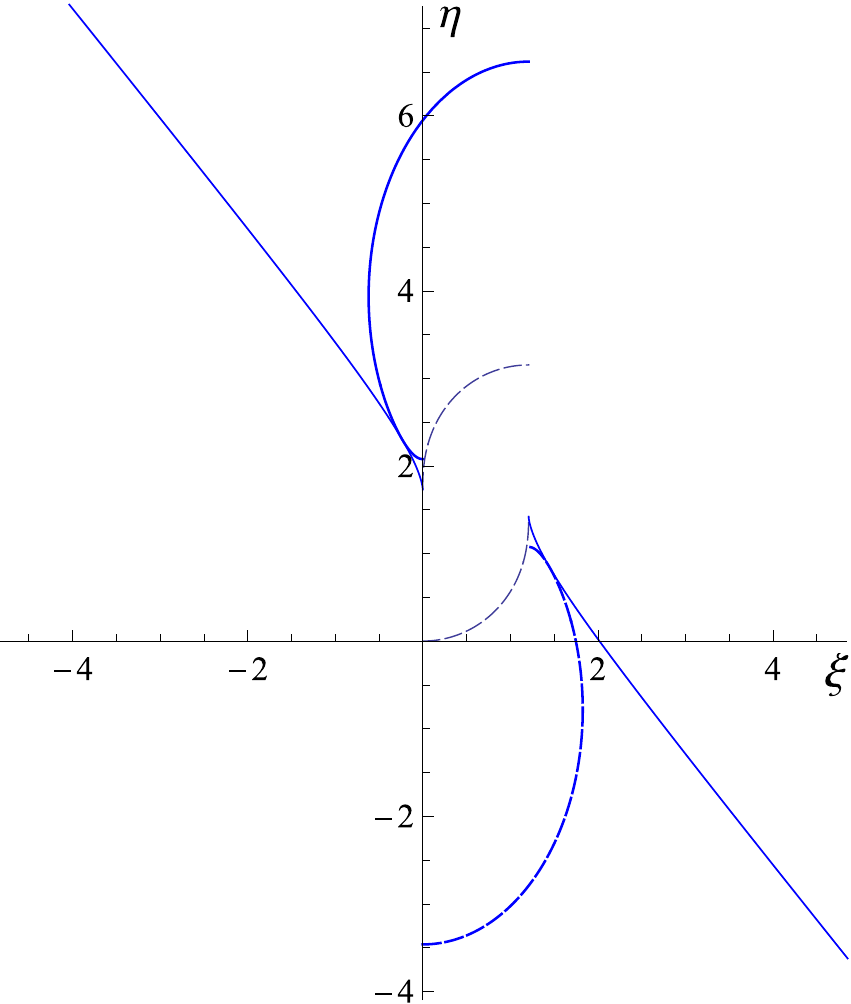}\\
  \vspace{-0.2cm}
  \caption{\label{Abw_1_5__-0_5_a}$\mathcal{C}_{-0.5}^*$ (thick, dashed), $\mathcal{C}_{1.5}^*$ (thick), $\R^*$ (thin), $\mathcal{C}_0^*$, $\mathcal{C}_1^*$; $0\leq t\leq 2\pi/3$}
\end{center}
\begin{center}
  \hspace{0.06cm}\includegraphics[scale=0.93]{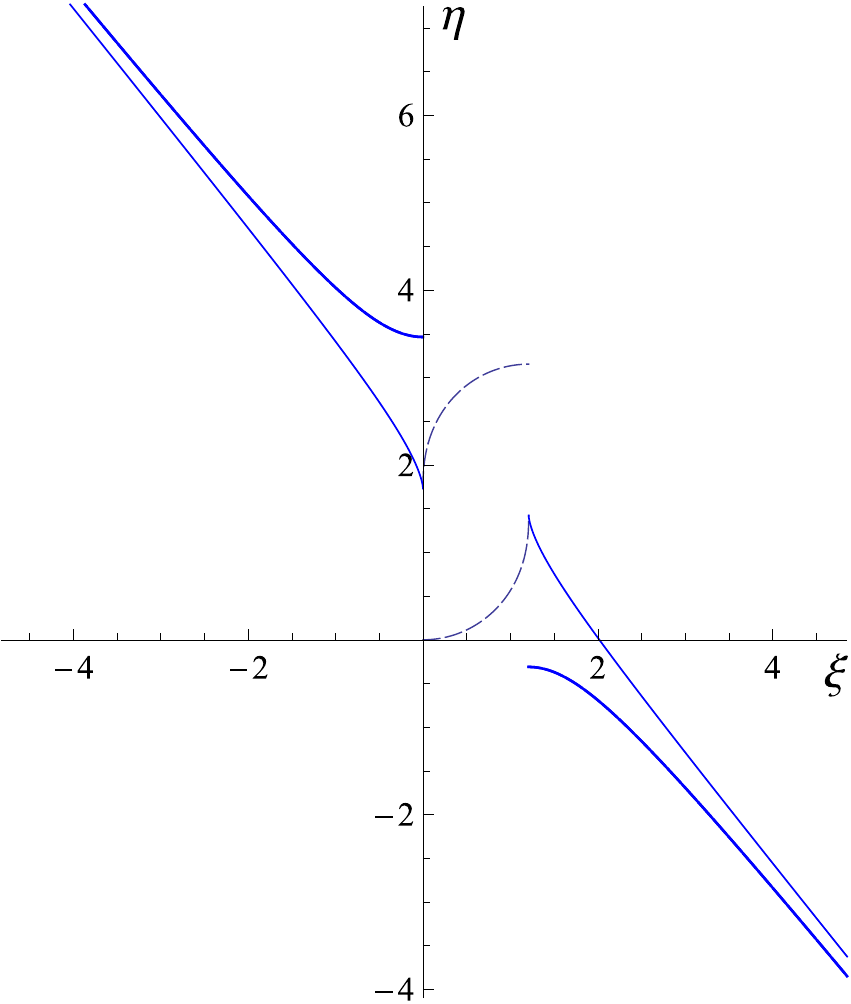}\\
  \vspace{-0.2cm}
  \caption{\label{Abw_inf_a} $\mathcal{C}_{\pm\infty}^*$ (thick), $\R^*$ (thin), $\mathcal{C}_0^*$ and $\mathcal{C}_1^*$ (dashed); $0\leq t\leq 2\pi/3$}
\end{center}
\end{figure}

\begin{figure}[h]
\begin{center}
  \vspace{0.04cm}
  \includegraphics[scale=0.93]{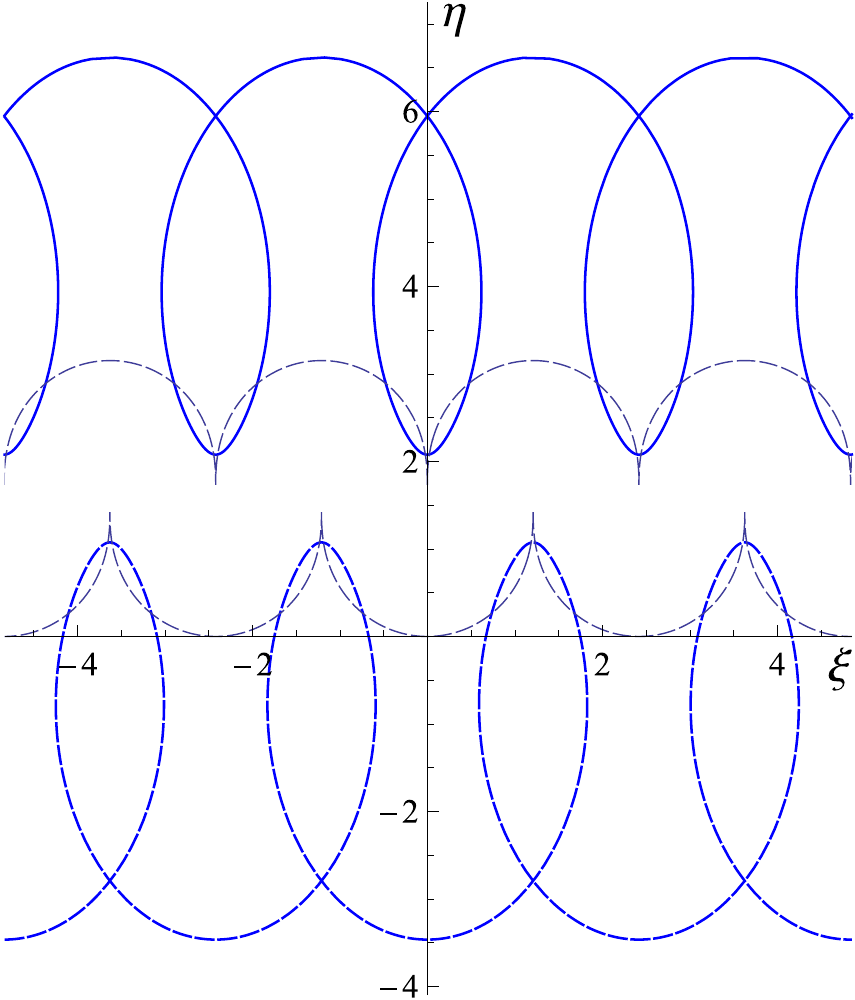}\\
  \vspace{-0.2cm}
  \caption{\label{Abw_1_5__-0_5} $\mathcal{C}_{-0.5}^*$ (thick, dashed), $\mathcal{C}_{1.5}^*$ (thick),  $\mathcal{C}_0^*$ and $\mathcal{C}_1^*$ (thin, dashed)} 
\end{center}
\begin{center}
  \includegraphics[scale=0.93]{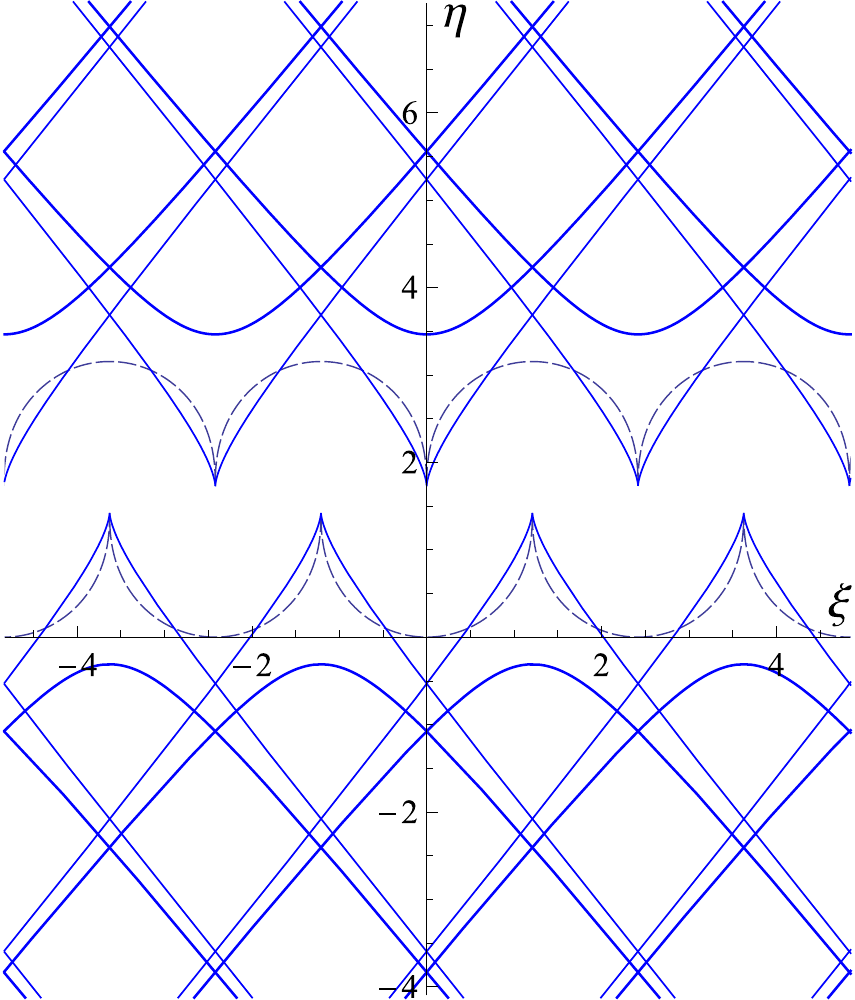}\\
  \vspace{-0.2cm}
  \caption{\label{Abw_inf} $\mathcal{C}_{\pm\infty}^*$ (thick), $\R^*$ (thin), $\mathcal{C}_0^*$ and $\mathcal{C}_1^*$ (dashed)}
\end{center}
\end{figure}

\clearpage

%% file: OOloid_L.tex
% !TeX root = OOloid_0.tex

\bigskip
\begin{center}
\begin{tabular}{c@{\qquad\quad}c}
Uwe B\"asel & Hans Dirnb\"ock\\[0.15cm]
HTWK Leipzig, Fakult\"at & Nussberg 22\\
Maschinenbau und Energietechnik, & 9062 Moosburg, Austria\\
Karl-Liebknecht-Stra{\ss}e 134, & \\
04277 Leipzig, Germany & \\[0.15cm]
{\small uwe.baesel@htwk-leipzig.de} & 
\end{tabular}
\end{center}